\documentclass[a4paper,11pt]{amsart}
\usepackage[left=2.7cm,right=2.7cm,top=3.5cm,bottom=3cm]{geometry}

\usepackage{amsthm,amssymb,amsmath,amsfonts,mathrsfs,amscd,dsfont}
\usepackage[latin1]{inputenc}
\usepackage[all,cmtip]{xy}
\usepackage{latexsym}
\usepackage{longtable}
\usepackage{mathtools}

\mathtoolsset{showonlyrefs}

\usepackage{xcolor}
\usepackage{graphicx}
\newcommand{\Bmu}{\mbox{$\raisebox{-0.59ex}
  {$l$}\hspace{-0.18em}\mu\hspace{-0.88em}\raisebox{-0.98ex}{\scalebox{2}
  {$\color{white}.$}}\hspace{-0.416em}\raisebox{+0.88ex}
  {$\color{white}.$}\hspace{0.46em}$}{}}

\numberwithin{equation}{section}

\newfont{\cyr}{wncyr10 scaled 1100}
\newfont{\cyrr}{wncyr9 scaled 1000}

\theoremstyle{plain}
\newtheorem{theorem}{Theorem}[section]
\newtheorem{proposition}[theorem]{Proposition}
\newtheorem{lemma}[theorem]{Lemma}
\newtheorem{corollary}[theorem]{Corollary}

\theoremstyle{definition}
\newtheorem{definition}[theorem]{Definition}
\newtheorem{assumption}[theorem]{Assumption}

\theoremstyle{remark}
\newtheorem{remark}[theorem]{Remark}

\newcommand{\Q}{\mathds Q}

\newcommand{\Z}{\mathds Z}

\newcommand{\F}{\mathds F}

\DeclareMathOperator{\Aut}{Aut}
\DeclareMathOperator{\Frob}{Frob}

\DeclareMathOperator{\Hom}{Hom}

\DeclareMathOperator{\Gal}{Gal}
\DeclareMathOperator{\GL}{GL}

\DeclareMathOperator{\Sel}{Sel}

\DeclareMathOperator{\im}{\mathrm{im}}
\DeclareMathOperator{\coker}{\mathrm{coker}}
\DeclareMathOperator{\corank}{\mathrm{corank}}
\DeclareMathOperator{\rank}{\mathrm{rank}}
\DeclareMathOperator{\length}{length}

\newcommand{\cyc}{{\rom{cyc}}}

\newcommand{\ord}{\mathrm{ord}}

\newcommand{\Sha}{\mbox{\cyr{X}}}

\newcommand{\longmono}{\mbox{\;$\lhook\joinrel\longrightarrow$\;}}

\newcommand{\longepi}{\mbox{\;$\relbar\joinrel\twoheadrightarrow$\;}}

\newfont{\gotip}{eufb10 at 12pt}

\newcommand{\cO}{{\mathcal O}}

\newcommand{\p}{\mathfrak{p}}

\newcommand{\rom}{\mathrm}

\include{thebibliography}

\begin{document}

\title[On Bloch--Kato Selmer groups and Iwasawa theory]{On Bloch--Kato Selmer groups and Iwasawa theory\\of $p$-adic Galois representations}
\author{Matteo Longo and Stefano Vigni}

\thanks{The authors are supported by PRIN 2017 ``Geometric, algebraic and analytic methods in arithmetic''.}

\begin{abstract}  A result due to R. Greenberg gives a relation between the cardinality of Selmer groups of elliptic curves over number fields and the characteristic power series of Pontryagin duals of Selmer groups over cyclotomic $\Z_p$-extensions at good ordinary primes $p$. We extend Greenberg's result to more general $p$-adic Galois representations, including a large subclass of those attached to $p$-ordinary modular forms of level $\Gamma_0(N)$ with $p\nmid N$. 
\end{abstract}

\address{Dipartimento di Matematica, Universit\`a di Padova, Via Trieste 63, 35121 Padova, Italy}
\email{mlongo@math.unipd.it}
\address{Dipartimento di Matematica, Universit\`a di Genova, Via Dodecaneso 35, 16146 Genova, Italy}
\email{stefano.vigni@unige.it}



\maketitle


\section{Introduction}

A classical result of R. Greenberg establishes a relation between the cardinality of Selmer groups of elliptic curves over number fields and the characteristic power series of Pontryagin duals of Selmer groups over cyclotomic $\Z_p$-extensions at good ordinary primes $p$. Our goal in this paper is to extend Greenberg's result to more general $p$-adic Galois representations, including a large subclass of those coming from $p$-ordinary modular forms of even weight and level $\Gamma_0(N)$ with $p\nmid N$. This generalization of Greenberg's theorem will play a role in our forthcoming paper on Kolyvagin's conjecture and a Bloch--Kato formula for higher (even) weight modular forms (\cite{LV-BSD}). 

Let us begin by recalling Greenberg's result. Let $E$ be an elliptic curve defined over a number field $F$, let $p$ be a prime number and suppose that $E$ has good ordinary reduction at all primes of $F$ above $p$. Moreover, assume that the $p$-Selmer group $\Sel_p(E/F)$ of $E$ over $F$ is finite. Let $F_\infty$ be the cyclotomic $\Z_p$-extension of $F$, set $\Gamma:=\Gal(F_\infty/F)\simeq\Z_p$ and write $\Lambda:=\Z_p[\![\Gamma]\!]$ for the corresponding Iwasawa algebra with $\Z_p$-coefficients, which we identify with the power series $\Z_p$-algebra $\Z_p[\![T]\!]$. Since $\Sel_p(E/F)$ is finite,  the $p$-primary Selmer group $\Sel_p(E/F_\infty)$ of $E$ over $F_\infty$ is $\Lambda$-cotorsion, \emph{i.e.}, the Pontryagin dual of $\Sel_p(E/F_\infty)$ is a torsion $\Lambda$-module (\cite[Theorem 1.4]{greenberg-cetraro}). Greenberg's result that is the starting point for the present paper (\cite[Theorem 4.1]{greenberg-cetraro}) states that if $f_E\in \Lambda$ is the characteristic power series of the Pontryagin dual of $\Sel_p(E/F_\infty)$, then
\begin{equation} \label{greenbergintro}
f_E(0)\sim \frac{\#\Sel_p(E/F)\cdot\prod_{\text{$v$ bad}}c_v(E)\cdot \prod_{v\mid p}\#\bigl(\tilde{E}_v(\F_v)\bigr)^2}{\#\bigl(E(F)_p\bigr)^2},
\end{equation}
where the symbol $\sim$ means that the two quantities differ by a $p$-adic unit, $c_v(E)$ is the Tamagawa number of $E$ at a prime $v$ of bad reduction, $\F_v$ is the residue field of $F$ at $v$, $\tilde{E}_v(\F_v)$ is the group of $\F_v$-rational points of the reduction $\tilde E_v$ of $E$ at $v$ and $E(F)_p$ is the $p$-torsion subgroup of the Mordell--Weil group $E(F)$. 

To the best of our knowledge, no generalization of this result is currently available to other settings of arithmetic interest, most notably that of modular forms of level $\Gamma_0(N)$ that are ordinary at a prime number $p\nmid N$ (see, however, \cite[Theorem 3.3.1]{JSW} for a result for Galois representations in an anticyclotomic imaginary quadratic context). In this article we offer a generalization of this kind. Now let us describe our main result more in detail.

Given a number field $F$ with absolute Galois group $G_F:=\Gal(\bar F/F)$ and an odd prime number $p$, we consider a $p$-ordinary (in the sense of Greenberg, \emph{cf.} \cite{greenberg-iwasawa}) representation 
\[ \rho_V:G_F\longrightarrow\Aut_K(V)\simeq\GL_r(K) \]
where $V$ is a vector space of dimension $r$ over a finite extension $K$ of $\Q_p$, equipped with a continuous action of $G_F$. We assume that $\rho_V$ is crystalline at all primes of $F$ above $p$, self-dual and unramified outside a finite set $\Sigma$ of primes of $F$ including those that either lie above $p$ or are archimedean. Writing $\cO$ for the valuation ring of $K$, we also fix a $G_F$-stable $\cO$-lattice $T$ inside $V$, set $A:=T\otimes_\cO(K/\cO)$ and assume that there exists a non-degenerate, skew-symmetric, Galois-equivariant pairing 
\[ T\times A\longrightarrow\Bmu_{p^\infty}, \]
where $\Bmu_{p^\infty}$ is the group of $p$-power roots of unity, so that $A$ and the Tate twist of the Pontryagin dual of $T$ become isomorphic. Finally, we impose on $V$ a number of technical conditions on invariant subspaces and quotients for the ordinary filtration at primes above $p$; the reader is referred to Assumption \ref{ass} for details. In particular, in \S \ref{sec2.3} we show that, choosing the prime number $p$ judiciously, these properties are enjoyed by the $p$-adic Galois representation attached by Deligne to a modular form of level $\Gamma_0(N)$ with $p\nmid N$.

As before, let $F_\infty$ be the cyclotomic $\Z_p$-extension of $F$ and set $\Gamma:=\Gal(F_\infty/F)$. Let $\Lambda:=\mathcal{O}[\![\Gamma]\!]$ be the Iwasawa algebra of $\Gamma$ with coefficients in $\cO$, which we identify with the power series $\cO$-algebra $\cO[\![T]\!]$. Finally, let $\Sel_\mathrm{Gr}(A/F_\infty)$ be the Greenberg Selmer group of $A$ over $F_\infty$ and let $\Sel_\mathrm{BK}(A/F)$ be the Bloch--Kato Selmer group of $A$ over $F$.  General control theorems due to Ochiai (\cite{Ochiai}) relate $\Sel_\mathrm{BK}(A/F)$ and $\Sel_\mathrm{Gr}(A/F_\infty)^\Gamma$, and one can think of our paper as a refinement of \cite{Ochiai} in which we describe in some cases the (finite) kernel and cokernel of the natural restriction map $\Sel_\mathrm{BK}(A/F)\rightarrow\Sel_\mathrm{Gr}(A/F_\infty)^\Gamma$. If $\Sel_\mathrm{BK}(A/F)$ is finite, then $\Sel_\mathrm{Gr}(A/F_\infty)$ is a $\Lambda$-cotorsion module, \emph{i.e.}, the Pontryagin dual $\Sel_\mathrm{Gr}(A/F_\infty)^\vee$ of $\Sel_\mathrm{Gr}(A/F_\infty)$ is a torsion $\Lambda$-module. Thus, when $\Sel_\mathrm{BK}(A/F)$ is finite we can consider the characteristic power series $\mathcal{F}\in\Lambda$ of $\Sel_\mathrm{Gr}(A/F_\infty)^\vee$. 

Our main result (Theorem \ref{main theorem}) is 

\begin{theorem} \label{thmintro}
If $\Sel_\mathrm{BK}(A/F)$ is finite, then $\mathcal{F}(0)\neq 0$ and 
\[ \#\bigl(\cO/\mathcal{F}(0)\cdot\cO\bigr)=\#\Sel_\mathrm{BK}(A/F)\cdot\prod_{\substack{v\in\Sigma\\v\nmid p}}c_v(A), \]
where $c_v(A)$ is the $p$-part of the Tamagawa number of $A$ at $v$. 
\end{theorem} 

The Tamagawa numbers $c_v(A)$ are defined in \S \ref{tamagawa-subsec}. It is worth pointing out that our local assumptions at primes of $F$ above $p$ ensure that the term corresponding to $E(F)_p$ in \eqref{greenbergintro} is trivial. Moreover, our conditions on the ordinary filtrations at primes $v$ of $F$ above $p$ force all the terms corresponding to $\tilde{E}_v(\F_v)$ in \eqref{greenbergintro} to be trivial as well. As will be apparent, our strategy for proving Theorem \ref{thmintro} is inspired by the arguments of Greenberg in \cite[\S 4]{greenberg-cetraro}.

We conclude this introduction with a couple of remarks of a general nature. First of all, several of our arguments can be adapted to other $\Z_p$-extensions $F_\infty/F$. However, in general 
this would require modifying the definition of Selmer groups at primes in $\Sigma$ that are not finitely decomposed in $F_\infty$. For example, suppose that $F$ is an imaginary quadratic field, $T$ is the $p$-adic Tate module of an elliptic curve over $\Q$ with good ordinary reduction at $p$ and $F_\infty$ is the anticyclotomic $\Z_p$-extension of $F$. Since prime numbers that are inert in $F$ split completely in $F_\infty$, some of the arguments described in this paper (\emph{e.g.}, the proof of Lemma \ref{corank}) fail and one needs to work with subgroups (or variants) of $\Sel_\mathrm{BK}(A/F)$ that are defined by imposing different conditions at primes in $\Sigma$ that are inert in $F$. For instance, the definition of Selmer groups in \cite[\S 2.2.3]{JSW} in the imaginary quadratic case requires that all the local conditions at inert primes be trivial, while in \cite{BD-IMC} one assumes ordinary-type conditions at those primes. Since the precise local conditions needed depend on the arithmetic situation being investigated, in this paper we chose to work with the cyclotomic $\Z_p$-extension of $F$ exclusively, thus considering only Bloch--Kato Selmer groups as defined below.  

Finally, we remark that our interest in $\Sel_\mathrm{Gr}(A/F_\infty)$ instead of the Bloch--Kato Selmer group $\Sel_\mathrm{BK}(A/F_\infty)$ of $A$ over $F_\infty$ is essentially motivated by the applications to \cite{LV-BSD} of the results in this paper. Actually, the results of Skinner--Urban on the Iwasawa main conjecture in the cyclotomic setting (\cite{SU}), which play a crucial role in \cite{LV-BSD}, are formulated in terms of $\Sel_\mathrm{Gr}(A/F_\infty)$ rather than $\Sel_\mathrm{BK}(A/F_\infty)$, which explains the focus of our article. However, alternative settings can certainly be considered; for example, \cite[Theorem 2.4]{Ochiai} establishes a relation between $\Sel_\mathrm{BK}(A/F)$ and $\Sel_\mathrm{BK}(A/F_\infty)^\Gamma$, and it would be desirable to prove a formula for the value at $0$ of the characteristic power series of $\Sel_\mathrm{BK}(A/F_\infty)$ analogous to that in Theorem \ref{thmintro}. 

\section{Galois representations} 

We fix the Galois representations that we consider in this paper, specifying our working assumptions. We will then show that these conditions are satisfied by a large class of $p$-ordinary crystalline representations attached to modular forms. 

\subsection{Notation and terminology} \label{sec2.1} 

To begin with, we introduce some general notation and terminology. If $p$ is a prime number and $M$ is a $\Z_p$-module $M$, then we write 
\[ M^\vee:=\Hom_{\Z_p}(M,\Q_p/\Z_p) \] 
for the Pontryagin dual of $M$. If $M$ is a module over the Galois group $\Gal(E/L)$ of some (Galois) field extension $E/L$, where $L$ is an extension of $\Q$ or $\Q_\ell$ for some prime number $\ell$, 
then we denote by $M(1)$ the Tate twist of $M$. Let $L$ be a local field of characteristic $0$, let $\bar L$ be a fixed algebraic closure of $L$ and let $G_L:=\Gal(\bar{L}/L)$ be the absolute Galois group of $L$. If $\Bmu_{p^\infty}\subset\bar L$ is the group of $p$-power roots of unity in $\bar L$, then local Tate duality 
\[ \langle\cdot,\cdot\rangle: H^i(G_L,M)\times H^{2-i}\bigl(G_L,M^\vee(1)\bigr)\longrightarrow H^2(G_L,\Bmu_{p^\infty})\simeq\Q_p/\Z_p \]
identifies $H^i(G_L,M)$ with $H^{2-i}\bigl(G_L,M^\vee(1)\bigr)$ for $i=0,1$. 

Let $F$ be a number field, let $\bar{F}$ be a fixed algebraic closure of $F$ and let $G_F:=\Gal(\bar{F}/F)$ be its absolute Galois group. For every prime $v$ of $F$ let $F_v$ be the completion of $F$ at $v$ and let $\mathcal{O}_v$ be the valuation ring of $F_v$. Moreover, let $G_{v}:=\Gal(\bar{F}_v/F_v)$ be the absolute Galois group of $F_v$ and let $I_v\subset G_v$ be the inertia subgroup. 

Let $T$ be a continuous $G_F$-module, which we assume to be free of finite rank $r$ over the valuation ring $\cO$ of a finite extension $K$ of $\Q_p$, where $p$ is an odd prime number. Fix a uniformizer $\pi$ of $\cO$ and set $\F:=\cO/(\pi)$ for the residue field of $K$, which is a finite field of characteristic $p$. Define $V:=T\otimes_\cO K$ and $A:=V/T=T\otimes_\cO(K/\cO)$, so that $T$ is an $\cO$-lattice inside $V$. Then $A\simeq(K/\mathcal{O})^r$ and $V\simeq K^r$ as groups and vectors spaces, respectively. Moreover, for every integer $n\geq 0$ there is an isomorphism of $G_F$-modules $T/\pi^nT\simeq A[\pi^n]$, where $A[\pi^n]$ is the $\pi^n$-torsion $\cO$-submodule of $A$. Set $T^*:=\Hom_{\Z_p}(A,\Bmu_{p^\infty})=A^\vee(1)$, $V^*:=T^*\otimes_{\cO}K$ and $A^*:=V^*/T^*=T^*\otimes_{\cO}K/\cO$, equipped with the induced actions of $G_F$.

The representation $V$ is \emph{self-dual} if there is a $G_F$-equivariant isomorphism $V\simeq V^*$. Let $V$ be self-dual and let us fix an isomorphism $\nu:V\overset\simeq\longrightarrow V^*$ as above. Suppose that $\nu(T)$ is homothetic to $T^*$, which means that there exists $\lambda\in K^\times$ such that $\lambda\cdot \nu(T)=T^*$ in $V^*$; this is an identification of $G_F$-modules, as the action of $G_F$ is $\mathcal{O}$-linear. The composition of $\nu$ with the multiplication-by-$\lambda$ map on $V^*$ induces an isomorphism of $G_F$-modules between the quotients $A=V/T$ and $A^*=V^*/T^*$.  

The representation $V$ is said to be \emph{ordinary} (in the sense of Greenberg) at a prime $v\,|\,p$ if there is a short exact sequence of representations
\[ 0\longrightarrow V^+\longrightarrow V\longrightarrow V^-\longrightarrow 0 \]
such that $I_v$ acts on $V^-$ via the cyclotomic character (see \cite{greenberg-iwasawa}). Set $r^+:=\dim_{K}(V^+)$ and $r^-:=\dim_K(V^-)$. Define $T^+:=V^+\cap T$ and $T^-:=V^-\cap T$, which are free $\cO$-submodules of $T$ of ranks $r^+$ and $r^-$, respectively. We also define 
$A^+:=V^+/T^+\simeq(K/\cO)^{r^+}$ and $A^-:=V^-/T^-\simeq(K/\cO)^{r^-}$. 

\subsection{Assumptions} \label{ass-subsec}

Notation being as in \S \ref{sec2.1}, write 
\[ \rho_V:G_F\longrightarrow\Aut_K(V)\simeq\GL_r(K) \]
and
\[ \rho_T:G_F\longrightarrow\Aut_\cO(T)\simeq\GL_r(\cO) \]
for the Galois representations associated with $V$ and $T$, respectively.

We work under the following assumption, which is slightly more restrictive than the one appearing in \cite{Ochiai}.

\begin{assumption}\label{ass}
\begin{enumerate} 
\item\label{condition 1} $\rho_V$ is unramified outside a finite set of primes of $F$; 
\item\label{condition 5} $\rho_V$ is crystalline at all primes $v\,|\,p$;
\item\label{condition 2} $\rho_V$ is self-dual;
\item\label{condition 6} $\rho_V$ is ordinary at all primes $v\,|\,p$; 
\item\label{condition 4} there is a skew-symmetric, $G_F$-equivariant and non-degenerate pairing 
\[ T\times A\longrightarrow\Bmu_{p^\infty} \] 
that induces a non-degenerate pairing 
\[ T/p^mT\times A[p^m]\longrightarrow\Bmu_{p^m} \] 
for every integer $m\geq1$, where $\Bmu_{p^m}$ is the group of $p^m$-th roots of unity (in particular, this gives an isomorphism $T^\vee(1)\simeq A$);   
\item\label{condition 7} for each prime $v\,|\,p$, one has:
\begin{enumerate}
\item\label{condition a} $H^0(F_v,A)=0$, 
\item\label{condition b} the largest cotorsion quotient of $H^0(I_v,A^-)$ is trivial,  
\item\label{condition c} $H^0\bigl({F_v},(T^-)^\vee(1)\bigr)=0$, 
\item\label{condition e} there exists a suitable basis such that $G_{F_v}$ acts on $A^+$ diagonally via non-trivial characters $\eta_1,\dots,\eta_{r^+}$ that do not coincide with the cyclotomic character. 
\end{enumerate}
\end{enumerate}
\end{assumption}

Let us define
\begin{equation} \label{sigma-set-eq}
\begin{split}
\Sigma:=\{\text{primes of $F$ at which $V$ is ramified}\}&\cup\{\text{primes of $F$ above $p$}\}\\
&\cup\{\text{archimedean primes of $F$}\}, 
\end{split}
\end{equation}
which is a finite set.

\begin{remark} \label{remark ass} 
We list some consequences of Assumption \ref{ass}. 
\begin{enumerate}
\item If the inertia group $I_v$ acts on $V^-$ via a non-trivial power of the cyclotomic character, then $H^0(I_v,V^-)=0$, so $H^0(I_v,T^-)=0$. Furthermore, $H^0(I_v,A^-)=0$ because $A^-=T^-\otimes_\mathcal{O}K/\mathcal{O}$. 
\item If $H^0(F_v,A)=0$, then the Galois-equivariant isomorphism $A\simeq T^\vee(1)$ ensures that $H^0\bigl(F_v,T^\vee(1)\bigr)=0$ as well. 
\item\label{remark3} If the isomorphism $V\simeq V^*$ induces an isomorphism $(T^\pm)^\vee(1)\simeq A^\mp$ (which is always true in the case of ordinary modular forms considered below), then the condition $H^0\bigl({F_v},(T^-)^\vee(1)\bigr)=0$ is equivalent to the condition $H^0(F_v,A^+)=0$, which is implied by \eqref{condition e}. 
\item\label{remark4} Since $V$ is ordinary at all $v\,|\,p$, it satisfies the so-called \emph{Panchishkin condition} at all such $v$ (\cite[Definition 2.2]{Ochiai}, \cite[\S 5.4]{panchishkin}).
\end{enumerate}
\end{remark}

\subsection{The case of modular forms} \label{sec2.3}

We want to check that if the prime number $p$ is chosen carefully, then Assumption \ref{ass} is satisfied by the $p$-adic Galois representation attached to a newform. Let $f(q)=\sum_{n\geq 1}a_n(f)q^n$ be a newform of weight $k\geq2$ and level $\Gamma_0(N)$. Let $\Q_f:=\Q\bigl(a_n(f)\mid n\geq1\bigr)$ be the Hecke field of $f$, which is a totally real number field. It is well known that the Fourier coefficients $a_n(f)$ are algebraic integers. 

Let $p$ be a prime number such that $p\nmid2N$ and fix a prime $\p$ of $\Q_f$ above $p$. We assume that
\begin{equation} \label{modular3}
\text{$\p$ is ordinary for $f$}. 
\end{equation}
In other words, we require $a_p(f)$ to be a $\p$-adic unit, \emph{i.e.}, $a_p(f)\in\cO^\times$. 

\begin{remark}
Let us say that a prime number $p$ is \emph{ordinary} for $f$ if $p\nmid a_p(f)$. Thanks to results of Serre on eigenvalues of Hecke operators (\cite[\S 7.2]{Serre-Cheb}), one can prove that if $k=2$, then the set of primes that are ordinary for $f$ has density $1$, so it is infinite (see, \emph{e.g.}, \cite[Proposition 2.2]{Fischman}). On the other hand, it is immediate to check that if $p$ is an ordinary prime for $f$ that is unramified in $\Q_f$, then there exists a prime $\wp$ of $\Q_f$ above $p$ such that $f$ is $\wp$-ordinary. As a consequence, a weight $2$ newform satisfies \eqref{modular3} with a suitable $\p$ for infinitely many primes $p$ (in fact, the set of such primes has density $1$). On the contrary, the existence of infinitely many ordinary primes for a modular form of weight larger than $2$ is, as far as we are aware of, still an open question.
\end{remark}

We also assume that
\begin{equation} \label{modular5} 
a_p(f)\not\equiv1\pmod{\p}.
\end{equation}
Write $K$ for the completion of $\Q_f$ at $\p$ and $\cO$ for the valuation ring of $K$. Let $V$ be the self-dual twist of the representation $V_{f,\p}$ of $G_\Q$ attached by Deligne to $f$ and $\p$ (\cite{Del-Bourbaki}), so that $V=V_{f,\p}(k/2)$. Choose a $G_\Q$-stable $\cO$-lattice $T\subset V$. The $\p$-adic representations 
\[ \rho_{f,\p}:G_\Q\longrightarrow \Aut_K(V)\simeq\GL_2(K) \]
and 
\[ \rho_{f,\p,T}:G_\Q\longrightarrow \Aut_\cO(T)\simeq\GL_2(\cO) \]
will play the roles of $\rho_V$ and $\rho_T$, respectively. In particular, $F=\Q$ in the notation of \S \ref{ass-subsec}. 

Now we show that Assumption \ref{ass} is satisfied by the representation $V$. First of all, it is well known that $V$ is unramified at all primes $\ell\nmid Np$ and crystalline at $p$: these properties correspond to conditions \eqref{condition 1} and \eqref{condition 5} in Assumption \ref{ass}. Furthermore, $V$ is the self-dual twist of $V_{f,\p}$, so \eqref{condition 2} in Assumption \ref{ass} is satisfied. On the other hand, \eqref{condition 4} in Assumption \ref{ass} corresponds to \cite[Proposition 3.1, (2)]{Nek}. 

Next, we show that \eqref{condition 6} and \eqref{condition 7} in Assumption \ref{ass} are satisfied. If $\ell\nmid Np$ and $\Frob_\ell$ is a geometric Frobenius at $\ell$, then the characteristic polynomial of $\rho_V(\Frob_\ell)$ is the Hecke polynomial $X^2-a_\ell(f)X+\ell^{k-1}$. Let $\alpha\in\cO^\times$ be the unit root of $X^2-a_p(f)X+p^{k-1}$, which exists because $f$ satisfies \eqref{modular3}, and let $\delta$ be the unramified character of the decomposition group $G_p:=\Gal(\bar{\Q}_p/\Q_p)$ given by $\delta(\Frob_p):=\alpha$, where $\Frob_p\in G_p/I_p$ is the arithmetic Frobenius. It is a classical result of Deligne and of Mazur--Wiles (see, \emph{e.g.}, \cite[Proposition 3.2]{Ochiai}) that the restriction of $V_{f,\p}$ to $G_p$ is equivalent to a representation of the form $\Big(\begin{smallmatrix}\delta^{-1}\chi_\cyc^{k-1}&c\\0&\delta\end{smallmatrix}\Big)$, where $c$ is a $1$-cocycle with values in $\cO$ and $\chi_\cyc:G_{\Q_p}\rightarrow\Z_p^\times$ is the $p$-adic cyclotomic character. It follows that the restriction of $V=V_{f,\p}(k/2)$ to $G_p$ is equivalent to a representation of the form $\Big(\begin{smallmatrix}\delta^{-1}\chi_\cyc^{k/2-1}&c\chi_\cyc^{-k/2}\\0&\delta\chi_\cyc^{-k/2}\end{smallmatrix}\Big)$. Thus, \eqref{condition 6} in Assumption \ref{ass} is satisfied and there is an exact sequence of $G_{\Q_p}$-modules 
\[ 0\longrightarrow V^+\longrightarrow V\longrightarrow V^-\longrightarrow0 \] 
such that $V^+$ and $V^-$ are $1$-dimensional $K$-vector spaces on which $G_{\Q_p}$ acts via $\delta^{-1}\cdot\chi_\cyc^{k/2-1}$ and $\delta\cdot\chi_\cyc^{-k/2}$, respectively. It follows that \eqref{condition e} in Assumption \ref{ass} is satisfied. Now, $I_{p}$ acts on $A^-\simeq K/\cO$ via the $(-k/2)$-th power of the cyclotomic character, which is non-trivial, so $H^0(I_{p},A^-)=0$, which shows that \eqref{condition b} in Assumption \ref{ass} holds. As for \eqref{condition c}, see part (3) of Remark \ref{remark ass}. 

Finally, we prove that \eqref{condition a} is satisfied. Suppose that $H^0(\Q_p,A)\neq 0$. Then one can find an element $a\in A[\p]$ such that $\sigma(a)=a$ for all $\sigma\in G_p$. Since $T/\p T\simeq A[\p]$, we may regard $a$ as an element of $T/\p T$. Choose a basis $\{e_1\}$ of the $1$-dimensional $\F$-vector space $T^+/\p T^+$ and complete it to a $\F$-basis $\{e_1,e_2\}$ of $T/\p T$. The action of $\sigma\in G_p$ on $T/\p T$ is then given by the matrix 
\[ \bar{\rho}_{f,\p}(\sigma)=\begin{pmatrix}\left(\delta^{-1}\cdot\chi_\cyc^{k/2-1}\right)(\sigma) & \left(c\cdot\chi_\cyc^{-k/2}\right)(\sigma)\\0 & \left(\delta\cdot\chi_\cyc^{-k/2}\right)(\sigma)\end{pmatrix}\pmod{\p}. \] 
Write $a=x_1e_1+x_2e_2$ with $x_1,x_2\in\F$. Let $v\,|\,p$ be a prime of $F$, set $\bar F_v:=\bar\Q_p$ and let $G_v:=\Gal(\bar{F}_v/F_v)\subset G_p$. The action of $\sigma\in G_v$ on $a$ is given by 
\begin{equation} \label{modularform}
\sigma(a)=\Bigl(x_1\bigl(\delta^{-1}\cdot\chi_\cyc^{k/2-1}\bigr)(\sigma)+x_2\bigl(c\cdot\chi_\cyc^{-k/2}\bigr)(\sigma)\Bigr)e_1+x_2\Bigl(\delta\cdot\chi_\cyc^{-k/2}\Bigr)(\sigma)e_2 \pmod{\p}.
\end{equation}
Then, in light of \eqref{modular5}, one easily shows from \eqref{modularform} that $\sigma(a)\neq a$. First suppose that $x_2=0$. Then $x_1\neq 0$ because $a\neq 0$, so it is enough to find $\sigma\in G_p$ 
such that $x_1\bigl(\delta^{-1}\cdot\chi_\cyc^{k/2-1}\bigr)(\sigma)\neq 1$. Let $\Frob_p\in\Gal(\Q_p^\mathrm{unr}/\Q_p)$ be an arithmetic Frobenius. If $k=2$, then $\Frob_p^j(a)=x_1\bar\alpha^{-j}e_1$ for all integers $j$, where $\bar\alpha=\alpha\pmod{\p}$. Since $\bar\alpha\neq 1$, one may find $j$ such that $x_1\bar\alpha^{-j}\neq 1$, and so $\sigma(a)\neq a$. If $k>2$, then a similar argument works. Namely, choose an integer $j$ such that $x_1\alpha^{-j}\neq 1$, pick a lift $F\in G_p$ of $\Frob_p^j$ and let $\bar{F}$ be the image of $F$ in $\Gal\bigl(\Q_p(\Bmu_{p^\infty})/\Q_p\bigr)$, so that $\chi_\cyc(F)=\chi_\cyc(\bar{F})$. 
If $x_1\alpha^{-j}\cdot\chi_\cyc^{k/2-1}(F)\neq 1$, then we are done; otherwise, since $x_1\alpha^{-j}\neq 1$, we see that $\chi_\cyc^{k/2-1}(\bar{F})\neq 1$, so the map $\sigma\mapsto \chi_\cyc^{k/2-1}$ is not the trivial character. It follows that $\chi_\cyc^{k/2-1}:\Gal\bigl(\Q_p(\Bmu_p)/\Q_p\bigr)\rightarrow\F_p^\times$ is surjective, and therefore we can find $\sigma\in I_p$ such that $\chi_\cyc^{k/2-1}(\sigma)\neq1$. Then
\[ \begin{split}
   x_1\bigl(\delta^{-1}\cdot\chi_\cyc^{k/2-1}\bigr)(F\sigma)&=x_1\alpha^{-j}\cdot\chi_\cyc^{k/2-1}(F\sigma)\\&=x_1\alpha^{-j}\cdot\chi_\cyc^{k/2-1}(F)\cdot\chi_\cyc^{k/2-1}(\sigma)=\chi_\cyc^{k/2-1}(\sigma)\neq1,
   \end{split} \] 
and we are done. The case $x_2\neq 0$ is similar; in fact, it suffices to show that there exists $\sigma$ such that $x_2\bigl(\delta\cdot\chi_\cyc^{-k/2}\bigr)(\sigma)\neq 1$. Choose $j$ such that $x_2\bar\alpha^{-j}\neq 1$ and fix a lift $F\in G_p$ of $\Frob_p^j$. If $x_2\alpha^{-j}\cdot\chi_\cyc^{-k/2}(F)\neq1$, then we are done, otherwise $\chi_\cyc^{-k/2}:\Gal\bigl(\Q_p(\Bmu_p)/\Q_p\bigr)\rightarrow\F_p^\times$ is surjective, so we can find $\sigma\in I_p$ such that $\chi_\cyc^{-k/2}(\sigma)\neq1$. It follows that
\[ \begin{split}
   x_2\bigl(\delta\cdot\chi_\cyc^{-k/2}\bigr)(F\sigma)&=x_2\alpha^{-j}\cdot\chi_\cyc^{-k/2}(F\sigma)\\
   &=x_2\alpha^{-j}\cdot\chi_\cyc^{-k/2}(F)\cdot\chi_\cyc^{-k/2}(\sigma)=\chi_\cyc^{-k/2}(\sigma)\neq1, 
   \end{split} \] 
and the proof of \eqref{condition a} in Assumption \ref{ass} is complete.

\section{Selmer groups} \label{sec-Selmer} 

Notation from \S \ref{sec2.1} is in force; moreover, we work under Assumption \ref{ass}. As before, let $F_\infty$ be the cyclotomic $\Z_p$-extension of $F$. Set $\Gamma:=\Gal(F_\infty/F)\simeq\Z_p$, choose a topological generator $\gamma$ of $\Gamma$ and let $\Lambda:=\mathcal{O}[\![{\Gamma}]\!]$ be the Iwasawa algebra of $\Gamma$ with coefficients in $\cO$, which can be identified with the formal power series $\cO$-algebra $\mathcal{O}[\![T]\!]$ by sending $\gamma$ to $T+1$. For every integer $n\geq0$ write $F_n$ for the $n$-th layer of $F_\infty/F$, \emph{i.e.}, the unique extension $F_n$ of $F$ such that $F_n\subset F_\infty$ and $\Gal(F_n/F)\simeq\Z/p^n\Z$ (in particular, $F_0=F$). For every prime $v$ of $F_n$ denote by $F_{n,v}$ the completion of $F_n$ at $v$, let $G_{n,v}:=\Gal(\bar{F}_{n,v}/F_{n,v})$ be the absolute Galois group of $F_{n,v}$ and let $I_{n,v}\subset G_{n,v}$ be the inertia subgroup. If $n=0$, in the previous notation we have $G_v=G_{0,v}$ and $I_v=I_{0,v}$. We also set $G_{\infty,v}:=\Gal(\bar{F}_v/F_{\infty,v})$ 
and denote by $I_{\infty,v}$ its inertia subgroup. 
%

\subsection{Local conditions at $\ell\neq p$} 

Fix an integer $n\geq0$. Fix also a prime number $\ell\neq p$ and a prime $v\,|\,\ell$ of $F_n$. For $\star\in\{V,A\}$, define 
\[ H^1_\mathrm{un}(F_{n,v},\star):=\ker\Bigl(H^1(F_{n,v},\star)\longrightarrow H^1(I_{n,v},\star)\Bigr). \]
Here, as customary, $H^1(F_{n,v},\star)$ stands for $H^1(G_{n,v},\star)$. By functoriality, there is a map $H^1(F_{n,v},V)\rightarrow H^1(F_{n,v},A)$; set
\[ H^1_f(F_{n,v},V):=H^1_\mathrm{ur}(F_{n,v},V),\quad H^1_f(F_{n,v},A):=\im\Bigl(H^1_f(F_{n,v},V)\longrightarrow H^1(F_{n,v},A)\Bigr). \] 
The commutative diagram
\[\xymatrix{
0\ar[r]& 
H^1_\mathrm{ur}(F_{n,v},V)\ar[r]&
H^1(F_{n,v},V)\ar[r]\ar[d]&
H^1(I_{n,v},V)\ar[d]\\
0\ar[r]& 
H^1_\mathrm{ur}(F_{n,v},A)\ar[r]&
H^1(F_{n,v},A)\ar[r]&
H^1(I_{n,v},A)\\
}\] 
shows that $H^1_f(F_{n,v},A)\subset H^1_\mathrm{ur}(F_{n,v},A)$.

\subsection{Local Tamagawa numbers} \label{tamagawa-subsec}

For every prime $v$ of $F$ we introduce the $p$-part of the Tamagawa number of $A$ at $v$. As we shall see, the product of these integers will appear in our main result.

\begin{lemma} \label{finite-index-lemma}
The index $\bigl[H^1_\mathrm{ur}(F_v,A):H^1_f(F_v,A)\bigr]$ is finite. 
\end{lemma}

\begin{proof} See \cite[Lemma 3.5]{Rubin-ES}. \end{proof}

The following notion is well defined thanks to Lemma \ref{finite-index-lemma}.

\begin{definition} \label{def Tamagawa} 
Let $v$ be a prime of $F$. The integer 
\[ c_v(A):=\bigl[H^1_\mathrm{ur}(F_v,A):H^1_f(F_v,A)\bigr] \] 
is the \emph{$p$-part of the Tamagawa number of $A$ at $v$}. 
\end{definition}

Recall the finite set $\Sigma$ from \eqref{sigma-set-eq}.

\begin{lemma}\label{good reduction}
If $v\notin\Sigma$, then $c_v(A)=1$. 
\end{lemma} 

\begin{proof} Since $T$ is unramified outside $Np$, this is \cite[Lemma I.3.5, (iv)]{Rubin-ES}. \end{proof}

\subsection{Local conditions at $p$}

Fix an integer $n\geq0$ and let $v\,|\,p$ be a prime of $F_n$.  

\subsubsection{The Bloch--Kato condition} 

Let $\mathbf{B}_\mathrm{cris}$ be Fontaine's crystalline ring of periods. Define 
\begin{align*} 
H^1_f(F_{n,v},V)&:=\ker\Bigl(H^1(F_{n,v},V)\longrightarrow H^1\bigl(F_{n,v},V\otimes_{\Q_p}\mathbf{B}_\mathrm{cris}\bigr)\Bigr)\\
\intertext{and}
H^1_f(F_{n,v},A)&:=\im\Bigl(H^1_f(F_{n,v},V)\longrightarrow H^1(F_{n,v},A)\Bigr),
\end{align*}
where the second arrow is induced by the canonical map $H^1(F_{n,v},V)\rightarrow H^1(F_{n,v},A)$.

\subsubsection{The Greenberg condition} 

As in \S \ref{sec2.1}, let $T^+:=T\cap V^+$, $A^+:=V^+/T^+$, $A^-:=A/A^+$. For $\star\in\{V,A\}$, define
\[ H^1_\mathrm{ord}(F_{n,v},\star):=\ker\Bigl(H^1(F_{n,v},\star)\longrightarrow H^1(I_{n,v},\star^-)\Bigr), \]
the map being induced by restriction and the canonical projection $\star\twoheadrightarrow\star^-$.

\subsubsection{Comparison between local conditions} \label{comparison-subsubsec}

Let $\mathbf{B}_\mathrm{dR}$ be Fontaine's de Rham ring of periods. Define
\[ H^1_g(F_v,V):=\ker\Bigl(H^1(F_v,V)\longrightarrow H^1\bigl(F_v,V\otimes_{\Q_p}\mathbf{B}_\mathrm{dR}\bigr)\Bigr). \]
Since $\mathbf{B}_\mathrm{cris}$ is a subring of $\mathbf{B}_\mathrm{dR}$, there is an inclusion $H^1_f(F_v,V)\subset H^1_g(F_v,V)$. Since $V$ is crystalline at $p$, by \cite[Proposition 12.5.8]{Nek-Selmer}, one has $\mathbf{D}_{\mathrm{cris},F_v}(V^-)=0$ (here, as usual, $\mathbf{D}_{\mathrm{cris},F_v}(W):=(W\otimes_{\Q_p}\mathbf{B}_\mathrm{cris})^{G_{F_v}}$ for a $G_{F_v}$-representation $W$). Then it follows from \cite[Proposition 12.5.7, (2), (ii)]{Nek-Selmer} that 
\begin{equation} \label{H-eq1}
H^1_f(F_v,V)=H^1_g(F_v,V). 
\end{equation} 
Moreover, by a result of Flach (\cite[Lemma 2]{Flach}; see also \cite[Proposition 4.2]{Ochiai}), there is an equality 
\begin{equation} \label{H-eq2}
H^1_g(F_v,V)=H^1_\mathrm{ord}(F_v,V).
\end{equation}
Combining \eqref{H-eq1} and \eqref{H-eq2} then yields 
\begin{equation} \label{ord-fin}
H^1_f(F_v,V)=H^1_\mathrm{ord}(F_v,V).
\end{equation}
Finally, in light of \eqref{ord-fin}, the commutativity of the square
\[ \xymatrix{H^1(F_v,V)\ar[d]\ar[r]&H^1(F_v,A)\ar[d]\\H^1(I_v,V^-)\ar[r]&H^1(I_v,A^-)} \]
shows that $H^1_f(F_v,A)\subset H^1_\mathrm{ord}(F_v,A)$.

\subsection{Selmer groups} \label{selmer-sec}

Now we introduce Selmer groups in the sense of Bloch--Kato and of Greenberg. 

\subsubsection{The Bloch--Kato Selmer group} 

Fix an integer $n\geq0$. The \emph{Bloch--Kato Selmer group of $A$ over $F_n$} is 
\[ \Sel_{\mathrm{BK}}(A/F_{n}):=\ker\Biggl(H^1(F_n,A)\longrightarrow \prod_{v}\frac{H^1(F_{n,v},A)}{H^1_f(F_{n,v},A)}\Biggr), \]
where $v$ varies over all primes of $F_n$ and the arrow is induced by the localization maps. Moreover, define the \emph{Bloch--Kato Selmer group of $A$ over $F_\infty$} as
\[ \Sel_\mathrm{BK}(A/F_\infty):=\varinjlim_n\Sel_{\mathrm{BK}}(A/F_{n})\]
the direct limit being taken with respect to the usual restriction maps in Galois cohomology. 

\subsubsection{The Greeenberg Selmer group} 

Fix an integer $n\geq0$. The \emph{Greenberg Selmer group of $A$ over $F_n$} is
\[ \Sel_\mathrm{Gr}(A/F_n):=\ker\Biggl(H^1(F_n,A)\longrightarrow\prod_{v\nmid p}\frac{H^1(F_{n,v},A)}{H^1_\mathrm{ur}(F_{n,v},A)}\times\prod_{v\mid p}\frac{H^1(F_{n,v},A)}{H^1_\mathrm{ord}(I_{n,v},A)}
\Biggr), \]
where $v$ varies over all primes of $F_n$ and the arrow is induced by the localization maps. Moreover, define the \emph{Greenberg Selmer group of $A$ over $F_\infty$} as
\begin{equation} \label{greenberg-lim-eq}
\Sel_\mathrm{Gr}(A/F_\infty):=\varinjlim_n\Sel_{\mathrm{Gr}}(A/F_{n}), 
\end{equation}
the direct limit being taken again with respect to the restriction maps. 

\begin{remark} 
For Galois representations associated with modular forms, the Selmer group considered in \cite{SU} is the Greenberg Selmer group. In many cases, the \emph{strict} Selmer group, which is defined as the Greenbebrg Selmer group with the difference that the local condition at a prime $v$ above $p$ is taken to be the kernel 
of the map 
\[ H^1(F_n,A)\longrightarrow\frac{H^1(F_{n,v},A)}{H^1_\mathrm{ord}(F_{n,v},A)}, \] 
is equal to the Bloch--Kato Selmer group (see, \emph{e.g.}, \cite[(23)]{Howard-inv}). 
\end{remark}

%

\section{Characteristic power series} \label{characteristic-sec}

As before, let $F_\infty$ be the cyclotomic $\Z_p$-extension of $F$ and put $\Gamma:=\Gal(F_\infty/F)\simeq\Z_p$. 
With notation as in \eqref{greenberg-lim-eq}, set  
\[ S:=\Sel_\mathrm{Gr}(A/F_\infty). \]
Furthermore, let 
\[ X:=S^\vee=\Hom(S,\Q_p/\Z_p) \]
be the Pontryagin dual of $S$. By the topological version of Nakayama's lemma (see, \emph{e.g.}, \cite[p. 226, Corollary]{BH}), the $\Lambda$-module $X$ is finitely generated. 

\subsection{Invariants and coinvariants of Selmer groups}

In what follows, $S^\Gamma$ (respectively, $S_\Gamma$) denotes the $\cO$-module of $\Gamma$-invariants (respectively, $\Gamma$-coinvariants) of $S$.


\begin{proposition}\label{lemma4.1} 
If $\Sel_\mathrm{BK}(A/F)$ is finite, then $S^\Gamma$ is finite and $S$ is a cotorsion $\Lambda$-module. 
\end{proposition}

Recall that, by definition, $S$ is $\Lambda$-cotorsion if $X$ is $\Lambda$-torsion.

\begin{proof} By \cite[Theorem 2.4]{Ochiai}, which can be applied in our setting using \cite[Lemma 12.5.7]{Nek-Selmer}, $\Sel_\mathrm{BK}(A/F)$ is finite if and only if $\Sel_\mathrm{BK}(A/F_\infty)^\Gamma$ is. On the other hand, by \cite[Proposition 4.2]{Ochiai} combined with \eqref{ord-fin}, $\Sel_\mathrm{BK}(A/F_\infty)^\Gamma$ is finite if and only if $S^\Gamma$ is. Now fix a topological generator $\gamma$ of $\Gamma$. By Pontryagin duality, the finiteness of $S^\Gamma$ is equivalent to the finiteness of $X/(\gamma-1)X$. By \cite[p. 229, Theorem]{BH}, it follows that $X$ is $\Lambda$-torsion, which means that $S$ is a cotorsion $\Lambda$-module. \end{proof}

\begin{remark}
The proof of Proposition \ref{lemma4.1} actually shows that the finiteness of $\Sel_\mathrm{BK}(A/F)$ is equivalent to the finiteness of $S^\Gamma$. Moreover, by \cite[Theorem 3]{Flach} and \cite[Proposition 4.1, (1)]{Ochiai}, for every $n\geq0$ the finiteness of $\Sel_\mathrm{BK}(A/F_n)$ is equivalent to the finiteness of $\Sel_\mathrm{Gr}(A/F_n)$. 
\end{remark}

In particular, when $\Sel_\mathrm{BK}(A/F)$ is finite we can consider the characteristic power series $\mathcal F\in\cO[\![T]\!]$ of $X$.  

\begin{proposition} \label{lemma5.11}
If $S^\Gamma$ is finite, then $S_\Gamma$ is finite, $\mathcal{F}(0)\neq 0$ and $\#\bigl(\cO/\mathcal{F}(0)\cO\bigr)=\#S^\Gamma\big/\#S_\Gamma$.  
\end{proposition}

\begin{proof} Proceed as in the proof of \cite[Lemma 4.2]{greenberg-cetraro}, which only deals with the $\cO=\Z_p$ case but carries over to our more general setting in a straightforward way. \end{proof}

\begin{remark} \label{euler-rem}
Since $\Gamma$ has cohomological dimension $1$ and $S$ is a direct limit of torsion groups, $H^2(\Gamma,S)=0$; moreover, since $\Gamma\simeq\Z_p$, we have $H^1(\Gamma,S)=S_\Gamma$. It follows that 
\[ \#S^\Gamma\big/\#S_\Gamma=\#H^0(\Gamma,S)\big/\#H^1(\Gamma,S) \] 
is the Euler characteristic of $M$. 
\end{remark}

From now on we work under 

\begin{assumption} \label{ass2}
The group $\Sel_\mathrm{BK}(A/F)$ is finite.
\end{assumption}
 
In light of Assumption \ref{ass2}, it follows from Propositions \ref{lemma4.1} and \ref{lemma5.11}
that 
\begin{equation} \label{1 and 2}
\text{$S$ is $\Lambda$-cotorsion, $S^\Gamma$ and $S_\Gamma$ are finite, $\mathcal{F}(0)\neq 0$}
\end{equation}
and
\begin{equation} \label{greenberg-eq2.2}
\#\bigl(\cO/\mathcal{F}(0)\cdot\cO\bigr)=\frac{\# S^\Gamma}{\# S_\Gamma}.
\end{equation}
In the following sections we shall study the quotient on the right hand side of \eqref{greenberg-eq2.2}. 

\section{Relating $\Sel_\mathrm{BK}(A/F)$ and $S^\Gamma$} 
%
Recall that Assumption \ref{ass2} is in force. As remarked in \eqref{1 and 2}, $S^\Gamma$ is finite as well. There is a natural map of finite groups 
\begin{equation} \label{s-map-eq}
s:\Sel_\mathrm{BK}(A/F)\longrightarrow S^\Gamma.
\end{equation}
 Then
\begin{equation}\label{greenberg-eq1}
\# S^\Gamma = \frac{\#\Sel_\mathrm{BK}(A/F)\cdot\# \coker(s)}{\#\ker(s)}.\end{equation}
Our next goal is to study the orders of the kernel and of the cokernel of $s$. 

\subsection{The map $r$} \label{r-subsec}

Set
\[ \mathcal{P}_\mathrm{BK}(A/F):=\prod_{v}\frac{H^1(F_v,A)}{H^1_f(F_v,A)} \]
and for every integer $n\geq0$ set also 
\[ \mathcal{P}_\mathrm{Gr}(A/F_n):=\prod_{v\nmid p}\frac{H^1(F_{n,v},A)}{H^1_\mathrm{ur}(F_{n,v},A)}\times\prod_{v\mid p}\frac{H^1(F_{n,v},A)}{H^1_\mathrm{ord}(F_{n,v},A)}, \]
where products are taken over all primes of $F$ and of $F_n$, respectively. With this notation available, we can write 
\begin{align*}
\Sel_\mathrm{BK}(A/F)&=\ker\Bigl(H^1(F,A)\longrightarrow \mathcal{P}_\mathrm{BK}(A/F)\Bigr)\\
\intertext{and}
\Sel_\mathrm{Gr}(A/F_n)&=\ker\Bigl(H^1(F_n,A)\longrightarrow \mathcal{P}_\mathrm{Gr}(A/F_n)\Bigr).
\end{align*}
Finally, we define 
\begin{equation} \label{greenberg-infinity-eq}
\mathcal{P}_\mathrm{Gr}(A/F_\infty):=\varinjlim_n\mathcal P_\mathrm{Gr}(A/F_n)=\prod_{v\nmid p}\frac{H^1(F_{\infty,v},A)}{H^1_\mathrm{ur}(F_{\infty,v},A)}\times\prod_{v\mid p}\frac{H^1(F_{\infty,v},A)}{H^1_\mathrm{ord}(F_{\infty,v},A)}
\end{equation} 
where the direct limit is taken with respect to the restriction maps; we also note that if a prime of $F_n$ splits completely in $F_m$ for $m\geq n$, then the corresponding map is the diagonal embedding. By definition, there is an equality 
\[ \Sel_\mathrm{Gr}(A/F_\infty)=\ker\Bigl(H^1(F_\infty,A)\longrightarrow \mathcal{P}_\mathrm{Gr}(A/F_\infty)\Bigr), \]
By construction (see \S \ref{selmer-sec} and \eqref{greenberg-infinity-eq}), there are natural maps $\mathcal{P}_\mathrm{BK}(A/F)\rightarrow\mathcal{P}_\mathrm{Gr}(A/F)$ and $\mathcal{P}_\mathrm{Gr}(A/F)\rightarrow\mathcal{P}_\mathrm{Gr}(A/F_\infty)$, which produce a map
\[ r:\mathcal{P}_\mathrm{BK}(A/F)\longrightarrow\mathcal{P}_\mathrm{Gr}(A/F_\infty). \] 
For every prime $v$ of $F$ let $w$ be a prime of $F_\infty$ above $v$. The map $r$ is given by a product 
$r=\prod_{v,w}r_{v,w}$, where 
\[ r_{v,w}:\frac{H^1(F_v,A)}{H^1_f(F_v,A)}\longrightarrow\frac{H^1(F_{\infty,w},A)}{H^1_\mathrm{ur}(F_{\infty,w},A)} \] 
if $v\nmid p$, while
\[ r_{v,w}:\frac{H^1(F_v,A)}{H^1_f(F_v,A)}\longrightarrow\frac{H^1(F_{\infty,w},A)}{H^1_\mathrm{ord}(F_{\infty,w},A)} \]
if $v\,|\,p$. Our next goal is to study the kernel of $r$ prime by prime.

\subsubsection{The map $r_{v,w}$ for $v\nmid p$} 

Assume that $v\nmid p$. Let $\Sigma$ be the finite set of primes of $F$ introduced in \eqref{sigma-set-eq}. We will distinguish two cases: $v\notin\Sigma$ and $v\in\Sigma$. 

\begin{lemma} \label{greenberg-lemma2.6} 
If $v\notin\Sigma$, then $r_{v,w}$ is injective. 
\end{lemma}

\begin{proof} By Lemma \ref{good reduction}, the kernel of $r_{v,w}$ is the kernel of the restriction map 
\begin{equation} \label{eq r}
\frac{H^1(F_v,A)}{H^1_\mathrm{ur}(F_v,A)}\longrightarrow 
\frac{H^1(F_{\infty,w},A)}{H^1_\mathrm{ur}(F_{\infty,w},A)}.
\end{equation}
With self-explaining notation, there are injections
\[ \frac{H^1(F_v,A)}{H^1_\mathrm{ur}(F_v,A)}\longmono H^1(I_v,A)=\Hom(I_v,A) \]
and
\[ \frac{H^1(F_{\infty,w},A)}{H^1_\mathrm{ur}(F_{\infty,w},A)}\longmono H^1(I_{\infty,w},A)=\Hom(I_{\infty,w},A), \]
where the equalities are a consequence of the fact that $A$ is unramified at $v$. Therefore, the kernel 
of \eqref{eq r} is contained in the kernel of the natural map 
\begin{equation} \label{inertia-hom-eq}
\Hom(I_v,A)\longrightarrow \Hom(I_{\infty,w},A).
\end{equation} 
Since $v$ is unramified, $F_{\infty,w}$ is a $\Z_p$-subextension of the unramified $\Z_p$-extension of $F_v$. Thus, $I_v=I_{\infty,w}$ and \eqref{inertia-hom-eq} is injective, which concludes the proof. \end{proof}

It follows from Lemma \ref{greenberg-lemma2.6} that $\ker(r)$ is the subgroup of $\mathcal{P}_\mathrm{BK}(A/F)$ consisting of elements $s$ such that $r_{v,w}(s)=0$ for all $w\,|\,v$ with $v\in\Sigma$. Thus, upon setting  
\[ \mathcal{P}_{\mathrm{BK}}^\Sigma(A/F):=\prod_{v\in\Sigma}\frac{H^1(F_v,A)}{H^1_f(F_v,A)} \]
and
\[ \mathcal{P}^\Sigma_\mathrm{Gr}(A/F_\infty):=\prod_{\substack{w\mid v\\v\in\Sigma,\,v\,\nmid\,p}}\frac{H^1(F_{\infty,w},A)}{H^1_\mathrm{ur}(F_{\infty,w},A)}\times\prod_{w\mid v\mid p}\frac{H^1(F_{\infty,w},A)}{H^1_\mathrm{ord}(F_{\infty,w},A)}, \]
it follows that $\ker(r)\subset\mathcal{P}_\mathrm{BK}^\Sigma(A/F)$; more precisely, $\ker(r)$ coincides with the kernel of the restriction map 
\begin{equation} \label{g-map-eq}
g:\mathcal{P}_\mathrm{BK}^\Sigma(A/F)\longrightarrow\mathcal{P}^\Sigma_\mathrm{Gr}(A/F_\infty)^\Gamma. 
\end{equation} 
We conclude that 
\begin{equation} \label{greenberg-eq2.8}
\#\ker(r)=\prod_{\substack{w\mid v\\v\in\Sigma}}\#\ker(r_{v,w}).
\end{equation}

\begin{lemma} \label{greenberg-lemma2.14} 
If $v\in\Sigma$ and $v\nmid p$, then $\#\ker(r_{v,w})=c_v(A)$. 
\end{lemma}

\begin{proof} The map $r_{v,w}$ splits as a composition
\[ r_{v,w}:H^1(F_v ,A)\big/H^1_f(F_v,A)\longrightarrow H^1(F_v ,A)\big/H^1_\mathrm{ur}(F_v,A)\longrightarrow H^1(F_{\infty,w},A)\big/H^1_\mathrm{ur}(F_{\infty,w},A). \]
Thus, since $\#\bigl(H^1_\mathrm{ur}(F_v ,A)/H^1_f(F_v ,A)\bigr)=c_v(A)$, it suffices to show that the map
\begin{equation} \label{greenberg-2.6/1}
H^1(F_v ,A)\big/H^1_\mathrm{ur}(F_v,A)\longrightarrow H^1(F_{\infty,w},A)\big/H^1_\mathrm{ur}(F_{\infty,w},A)
\end{equation}
is injective. Let $F_v^\mathrm{ur}$ be the maximal unramified extension of $F_v$ and set $G_v^\mathrm{ur}:=\Gal(F_v ^\mathrm{ur}/F_v)$. There is a commutative diagram 
\begin{equation} \label{inertia-diagram-eq}
\xymatrix{
0\ar[r]& 
H^1_\mathrm{ur}(F_v ,A)\ar[r]\ar[d]&
H^1(F_v ,A)\ar[r]\ar[d]&
H^1(I_v ,A)^{G_v^\mathrm{ur}}\ar[r]\ar[d]&
0\\
0\ar[r]&
H^1_\mathrm{ur}(F_{\infty,w},A)\ar[r]&
H^1(F_{\infty,w},A)\ar[r]&
H^1(I_{\infty,w},A)
}
\end{equation}
with exact rows. Notice that, since $G_v^\mathrm{ur}\simeq\prod_\ell\Z_\ell$, the surjectivity of the right non-trivial map in the top row of \eqref{inertia-diagram-eq} stems from the vanishing of $H^2(G_v^\mathrm{ur},A)$ (\cite[Proposition 1.4.10, (2)]{Rubin}). It follows that the kernel of the map in \eqref{greenberg-2.6/1} can be identified with the kernel of the rightmost vertical map in \eqref{inertia-diagram-eq}, which is isomorphic (by the inflation-restriction exact sequence) to $H^1\bigl(I_v /I_{\infty,w},A^{I_\infty,w}\bigr)^{G_v^\mathrm{ur}}$. Since $F_{\infty,w}/F_v $ is unramified if $v\nmid p$, we have $I_v =I_{\infty,w}$, so $H^1\bigl(I_v /I_{\infty,w},A^{I_\infty,w}\bigr)=0$. It follows that \eqref{greenberg-2.6/1} is injective, which completes the proof. \end{proof}

\subsubsection{The map $r_{v,w}$ for $v\,|\,p$}  

Now we study the local conditions at a prime $v\,|\,p$. Recall that, by \eqref{ord-fin}, we have $H^1_\mathrm{ord}(F_v,V)=H^1_f(F_v,V)$. Moreover, as explained in \S \ref{comparison-subsubsec}, there is an inclusion $H^1_f(F_v,A)\subset H^1_\mathrm{ord}(F_v,A)$. 

%

The lemma below, whose proof uses in a crucial way the triviality of local invariants from part \eqref{condition a} in Assumption \ref{ass}, forces the terms corresponding to $\tilde{E}_v(\F_v)$ in \eqref{greenbergintro} to be trivial for all primes $v$ of $F$ above $p$.

\begin{lemma} \label{greenberg-lemma2.15}  
If $v\,|\,p$, then $r_{v,w}$ is injective. 
\end{lemma}

\begin{proof} To begin with, note that the map $r_{v,w}$ can be written as the composition  
\begin{equation} \label{composition}
r_{v,w}:\frac{H^1(F_v,A)}{H^1_f(F_v,A)}\longrightarrow
\frac{H^1(F_v,A)}{H^1_\mathrm{ord}(F_v,A)}\longrightarrow 
\frac{H^1(F_{\infty,w},A)}{H^1_\mathrm{ord}(F_{\infty,w},A)},
\end{equation}
where the first arrow is induced by the identity map of $H^1(F_v,A)$ and the second by the obvious (restriction) map $H^1(F_v,A)\rightarrow H^1(F_{\infty,w},A)$. Our strategy is to prove that both maps appearing in \eqref{composition} are injective. As we shall see, the proof of the injectivity of the second map is, thanks to Assumption \ref{ass}, straightforward, while dealing with the first map is much more delicate.

We first take care of the second map in \eqref{composition}. Let us consider the commutative diagram with exact rows
\[ \xymatrix{
0\ar[r]& 
H^1_\mathrm{ord}(F_v,A)\ar[r]\ar[d]&
H^1(F_v,A)\ar[r]\ar[d]&
H^1(I_v,A^-)\ar[d]
\\
0\ar[r]&
H^1_\mathrm{ord}(F_{\infty,w},A)\ar[r]&
H^1(F_{\infty,w},A)\ar[r]&
H^1(I_{\infty,w},A^-).
}
\]
The kernel of the rightmost vertical arrow is isomorphic (by the inflation-restriction exact sequence) to 
$H^1\bigl(I_p/I_{\infty,w},(A^-)^{I_\infty,w}\bigr)$. By Assumption \ref{ass}, we know that $(A^-)^{I_\infty,w}=0$ (\emph{cf.} part (1) of Remark \ref{remark ass}), so the second arrow in \eqref{composition} is injective. It follows that $\ker(r_{v,w})$ is just the kernel of the first map in \eqref{composition}. We tackle the study of this map by adapting arguments from \cite{Flach} and \cite[Proposition 4.2]{Ochiai}. Of course, proving that the above-mentioned map is injective amounts to showing that $H^1_f(F_v,A)=H^1_\mathrm{ord}(F_v,A)$. 

Let us consider the commutative diagram with exact rows
\[ \xymatrix{
&&
\displaystyle{\frac{H^1(F_v,T)}{H^1(F_v,T)_\mathrm{tors}}}\ar[d]\ar[r]^-a&
\biggl(\displaystyle{\frac{H^1(I_v,T)}{H^1(I_v,T)_\mathrm{tors}}}\biggr)^{\!\!G_v^\mathrm{ur}}\ar[d]\\
0\ar[r]&
H^1_f(F_v,V)=H^1_\mathrm{ord}(F_v,V)\ar[r]\ar[d]\ar@/^3.0pc/[dd]^-{g}&
H^1(F_v,V)\ar[r]^-{b}\ar[dd]&
H^1(I_v,V^-)^{G_v^\mathrm{ur}}\ar[dd]\\
&
H^1_f(F_v,A)\ar@{^(->}[d]&
&
\\
0\ar[r]&
H^1_\ord(F_v,A)\ar[r]^-d&
H^1(F_v,A)\ar[r]^-c\ar[d]&
H^1(I_v,A^-)^{G_v^\mathrm{ur}}\ar[d]\\
&&
H^2(F_v,T)_\mathrm{tors}\ar[r]& H^2(I_v,T^-)_\mathrm{tors}^{G_v^\mathrm{ur}}.
}
\]
The map $b$ splits as a composition 
\[ b:H^1(F_v,V)\overset{b'}\longrightarrow H^1(F_v,V^-)\overset{b''}\longrightarrow H^1(I_v,V^-)^{G_v^\mathrm{ur}}. \]
The cokernel of $b'$ injects into $H^2(F_v,V^+)$, which is isomorphic to $H^0\bigl(F_v,(V^-)^\vee(1)\bigr)$ by local Tate duality. Since the $H^0\bigl(F_v,(T^-)^\vee(1)\bigr)=0$ by part \eqref{condition c} of Assumption \ref{ass}, the latter group is trivial. On the other hand, the cokernel of $b''$ injects into $H^2\bigl(G_v^\mathrm{ur},(V^-)^{I_v}\bigr)$, which is trivial because $V^-$ has trivial $I_v$-invariants by Assumption \ref{ass} (\emph{cf.} part (1) of Remark \ref{remark ass}). Therefore, the map $b$ is surjective. Moreover, local Tate duality identifies $H^2(F_v,T)$ with $H^0\bigl(F_v,T^\vee(1)\bigr)$, which is trivial by Assumption \ref{ass} (\emph{cf.} part (2) of Remark \ref{remark ass}), hence $H^2(F_v,T)=0$. The snake lemma then gives an isomorphism $\coker(a)\simeq\coker(g)$.
Now we study the cokernel of $a$. Let us consider the commutative diagram with exact rows 
\[\xymatrix{
0\ar[r]& 
H^1(F_v,T)_\mathrm{tors}\ar[r]\ar[d]&
H^1(F_v,T)\ar[r]\ar[d]^h &
\displaystyle{\frac{H^1(F_v,T)}{H^1(F_v,T)_\mathrm{tors}}}\ar[r]\ar[d]^a &0\\
0\ar[r]&
H^1(I_v,T^-)_\mathrm{tors}^{G_v^\mathrm{ur}}\ar[r]&
H^1(I_v,T^-)^{G_v^\mathrm{ur}}\ar[r]^-e&
\biggl(\displaystyle{\frac{H^1(I_v,T^-)}{H^1(I_v,T^-)_\mathrm{tors}}}\biggr)^{\!\!G_v^\mathrm{ur}}.
}
\]
The cokernel of $e$ is a subgroup of $H^1\bigl(G_v^\mathrm{ur},H^1(I_v,T^-)_\mathrm{tors}\bigr)$. The group $H^1(I_v,T^-)_\mathrm{tors}$ is (isomorphic to) the largest cotorsion quotient of $H^0(I_v,A^-)$, which is trivial by part \eqref{condition b} of Assumption \ref{ass}. Thus, $H^1\bigl(G_v^\mathrm{ur},H^1(I_v,T^-)_\mathrm{tors}\bigr)$ is trivial too, and so is $\coker(e)$. It follows that the natural map $\coker(h)\rightarrow \coker(a)$ is surjective. On the other hand, the map $h$ can be written as 
\[ h:H^1(F_v,T)\overset{h'}\longrightarrow H^1(F_v,T^-)\overset{h''}\longrightarrow H^1(I_v,T^-)^{G_v^\mathrm{ur}}. \]
The cokernel of $h'$ injects into $H^2(F_v,T^+)$, whose dual $H^0\bigl(F_v,(T^-)^\vee(1)\bigr)$ is trivial thanks to part \eqref{condition c} of Assumption \ref{ass}. Moreover, the cokernel of $h''$ injects into $H^2\bigl(G_v^\mathrm{ur},(T^-)^{I_v}\bigr)$, which is trivial because $(T^-)^{I_v}=0$ by Assumption \ref{ass} (\emph{cf.} part (1) of Remark \ref{remark ass}). Thus, the map $h$ is surjective, and we conclude that the cokernel of $a$ is trivial. Since $\coker(a)\simeq\coker(g)$, it follows that the cokernel of $g$ is trivial as well. This means that $g$ is surjective, \emph{i.e.}, $H^1_f(F_v,A)=H^1_\mathrm{ord}(F_v,A)$, as was to be shown. \end{proof}

If follows from a combination of equality \eqref{greenberg-eq2.8} and Lemmas \ref{greenberg-lemma2.14} and \ref{greenberg-lemma2.15} that 
\begin{equation}\label{ker(r)}
\#\ker(r)=\prod_{\substack{v\in\Sigma\\v\neq p}}c_v(A).
\end{equation}

\subsection{Surjectivity of restriction} \label{surjectivity-subsec}

Denote by $F_\Sigma$ the maximal extension of $F$ unramified outside $\Sigma$, and for any $\Gal(F_\Sigma/F)$-module $M$ set $H^i(F_\Sigma/F,M):=H^i\bigl(\Gal(F_\Sigma/F),M\bigr)$. The symbol $H^i(F_\Sigma/F_\infty,M)$ will have an analogous meaning. It follows from Lemma \ref{greenberg-lemma2.6} that 
\[ \Sel_{\mathrm{BK}}(A/F)=\ker\Bigl(H^1(F_\Sigma/F,A)\overset{\delta_\Sigma}\longrightarrow\mathcal{P}_\mathrm{BK}^\Sigma(A/F)\Bigr) \]
and
\[ \Sel_\mathrm{Gr}(A/F_\infty)=\ker\Bigl(H^1(F_\Sigma/F_\infty,A)\xrightarrow{\delta_{\infty,\Sigma}}\mathcal{P}^\Sigma_\mathrm{Gr}(A/F_\infty)\Bigr), \]
where $\delta_\Sigma$ and $\delta_{\infty,\Sigma}$ are the restriction maps. For notational convenience, if $M$ is a $G_F$-module and $F'$ is an algebraic extension of $F$, then we set $M(F'):=H^0(F',M)$.

The following lemma, whose proof uses (a global consequence of) part \eqref{condition a} of Assumption \ref{ass}, implies that the term corresponding to $E(F)_p$ in \eqref{greenbergintro} is trivial.

\begin{lemma} \label{greenberg-lemma2.8} 
$A(F_\infty)=0$.
\end{lemma}

\begin{proof} Since $A=\varinjlim_nA[\pi^n]$, it suffices to show that $A[\pi](F_\infty)=0$. Recall that $A[\pi]$ is a finite-dimensional vector space over the finite field $\F$ of characteristic $p$; in particular, $A[\pi]$ is finite. It follows that if $x\in A[\pi]$, then the stabilizer of $x$ in $G_F$ is a closed and finite index subgroup of $G_F$, so it is equal to $G_{F'}$ for a suitable finite extension $F'$ of $F$. If, moreover, $x\in A[\pi](F_\infty)$, then $F'\subset F_\infty$, hence $F'=F_n$ for a suitable $n\geq0$ and $x\in A[\pi](F_n)$. Thus, we are reduced to showing that $A[\pi](F_n)=0$ for all $n\geq0$. Since $A[\pi](F)=0$ by part \eqref{condition a} of Assumption \ref{ass} and $[F_{m+1}:F_m]=p$ for all $m\geq0$, the triviality of $A[\pi](F_n)$ for all $n\geq0$ follows from \cite[Proposition 26]{Serre} by induction on $n$. \end{proof}

\begin{remark}
If we imposed additional assumptions on the Galois representation $\rho_V$, then we could avoid using part \eqref{condition a} of Assumption \ref{ass} to deduce that $A(F_\infty)=0$. For example, if we required the reduction of $\rho_T$ modulo $\p$ to be irreducible with non-solvable image (which in the case of non-CM modular forms is true for all but finitely many $p$, \emph{cf.} \cite[Lemma 3.9]{LV}, \cite[\S 2]{ribet2}), then we could show that $A(F_\infty)$ is trivial by proceeding as in the proof of \cite[Lemma 2.4]{LV-kyoto} (see also \cite[Lemma 3.10, (2)]{LV}). However, the local vanishing from part \eqref{condition a} of Assumption \ref{ass} plays a much more delicate role in the proof of Lemma \ref{greenberg-lemma2.15}, which is the reason why we decided to assume this condition right from the outset.
\end{remark}

\begin{lemma} \label{surjectivity lemma} 
The map $\delta_\Sigma$ is surjective. 
\end{lemma}

\begin{proof} We apply some results from \cite[Section 4]{greenberg-cetraro}, so we first recall the setting of \cite{greenberg-cetraro}. Let $M$ be a $G_F$-module isomorphic to $(\Q_p/\Z_p)^d$ for some 
integer $d\geq 1$. As in \cite{greenberg-cetraro}, define $T^*:=\Hom_{\Z_p}(M,\mu_{p^\infty})=M^\vee(1)$, 
$V^*:=T^*\otimes_{\Z_p}\Q_p$ and $M^*:=V^*/T^*=T^*\otimes_{\Z_p}\Q_p/\Z_p$. Let
\[ H^1(F_v,M)\times H^1(F_v,T^*)\longrightarrow\Q_p/\Z_p \] 
be the local Tate pairing, which is perfect. Suppose that for each $v\in\Sigma$ we have a divisible subgroup $L_v$ of $H^1(F_v,M)$, then consider the Selmer group 
\[ \Sel(M/F):=\ker\biggl(H^1(F_\Sigma/F,M)\overset\gamma\longrightarrow\prod_{v\in\Sigma}\frac{H^1(F_v,M)}{L_v}\biggr). \]
Furthermore, denote by $U_v^*$ the orthogonal complement of $L_v$ under the local Tate pairing. 
Write $L_v^*$ for the image in $H^1(F_v,M^*)$ of the $\Q_p$-subspace of $H^1(F_v,V^*)$ generated by the image of $U_v^*$ under the natural map $H^1(F_v,T^*)\rightarrow H^1(F_v,V^*)$. Set
\[ \Sel(M^*/F):=\ker\biggl(H^1(F_\Sigma/F,M^*)\longrightarrow\prod_{v\in\Sigma}\frac{H^1(F_v,M^*)}{L_v^*}\biggr). \]
The arguments in \cite[pp. 100--101]{greenberg-cetraro} show that if $\Sel(M^*/F)$ is finite and $H^0(F_\Sigma/F,M^*)$ is trivial, then $\gamma$ is surjective. 

Now we return to our setting, with $M=A$, $T$ and $V$ as in Assumption \ref{ass}. To apply the results explained above, we recall that, by Assumption \ref{ass}, the image of $T$ in $V^*= V^\vee(1)$ under the isomorphism $V\simeq V^*$ is homothetic to $T^*$. Furthermore, the Bloch--Kato conditions $H^1_f(F_v,A)$ and $H^1_f(F_v,A^*)$ are orthogonal under local Tate duality (\cite[Proposition 3.8]{BK}). Thus, $\Sel_\mathrm{BK}(A/F)\simeq\Sel_\mathrm{BK}(A^*/F)$, so the finiteness of $\Sel_\mathrm{BK}(A/F)$ is equivalent to that of $\Sel_\mathrm{BK}(A^*/F)$. To conclude the proof we only need to check that $H^0(F_\Sigma/F,A^*)$ is trivial. Since $T$ and $T^*$ are homothetic, and $A=V/T$, $A^*=V^*/T^*$, it is enough to show that $H^0(F_\Sigma/F,A)$ is trivial. Since $A$ is unramified outside $\Sigma$, we have $A=H^0(F_\Sigma,A)$, so $H^0(F,A)=H^0(F_\Sigma/F,A)$. It follows from Lemma \ref{greenberg-lemma2.8} that $H^0(F_\Sigma/F_\infty,A)=0$ and, \emph{a fortiori}, $H^0(F_\Sigma/F,A)=0$. \end{proof}

\subsection{An application of the snake lemma} 

It follows from Lemma \ref{surjectivity lemma} that there is a commutative diagram 
\begin{equation} \label{greenberg-diagram}
\xymatrix@C=33pt@R=30pt{
0\ar[r]& 
\Sel_\mathrm{BK}(A/F)\ar[r]\ar[d]^-s&
H^1(F_\Sigma/F,A)\ar[r]^-{\delta_\Sigma}\ar[d]^-{h}&
\mathcal{P}_\mathrm{BK}^\Sigma(A/F)\ar[r]\ar[d]^-{g}&
0\\
0\ar[r]& 
S^\Gamma\ar[r]&
H^1(F_\Sigma/F_\infty,A)^\Gamma\ar[r]^-{\delta_{\infty,\Sigma}^\Gamma}&
\mathcal{P}^\Sigma_\mathrm{Gr}(A/F_\infty)^\Gamma
}
\end{equation}
with exact rows, where $s$ is as in \eqref{s-map-eq}, $h$ is restriction, $g$ is as in \eqref{g-map-eq} and $\delta_{\infty,\Sigma}^\Gamma$ is the map induced by $\delta_{\infty,\Sigma}$ between $\Gamma$-invariants.

\begin{lemma} \label{lemma2.3}
The map $h$ is an isomorphism and the map $s$ is injective.
\end{lemma}

\begin{proof} By inflation-restriction, there is an exact sequence 
\[ 0\longrightarrow H^1\bigl(\Gamma,A(F_\infty)\bigr)\longrightarrow H^1(F_\Sigma/F,A)\overset{h}\longrightarrow H^1(F_\Sigma/F_\infty,A)^\Gamma\longrightarrow H^2\bigl(\Gamma,A(F_\infty)\bigr), \]
and the lemma follows from Lemma \ref{greenberg-lemma2.8} and the injection $\ker(s)\hookrightarrow\ker(h)$. \end{proof}

Recall from \eqref{1 and 2} that $S^\Gamma$ is finite, so $\coker(s)$ is finite.

\begin{lemma} \label{lemma2.8}
$\#S^\Gamma=\#\Sel_\mathrm{BK}(A/F)\cdot\#\ker(g)$. 
\end{lemma}

\begin{proof} From \eqref{greenberg-diagram}, the snake lemma gives an exact sequence
\begin{equation}\label{greenberg-eq2}
\ker(h)\longrightarrow \ker(g)\longrightarrow 
\coker(s)\longrightarrow\coker(h).
\end{equation}
On the other hand, by Lemma \ref{lemma2.3}, both $\ker(h)$ and $\coker(h)$ are trivial, so \eqref{greenberg-eq2} yields an isomorphism $\ker(g)\simeq\coker(s)$. In particular, $\#\coker(s)=\#\ker(g)$, and the searched-for equality follows immediately from \eqref{greenberg-eq1}. \end{proof}

The proposition below provides a crucial step towards our main result.

\begin{proposition} \label{greenberg1}
$\# S^\Gamma= \#\Sel_\mathrm{BK}(A/F)\cdot\prod\limits_{\substack{v\in\Sigma\\v\nmid p}}c_v(A)$. 
\end{proposition}

\begin{proof}
Combine Lemma \ref{lemma2.8}, the equality $\ker(r)=\ker(g)$ and \eqref{ker(r)}. 
\end{proof}

\subsection{Finite index $\Lambda$-submodules} 

Our present goal is to generalize \cite[Proposition 4.9]{greenberg-cetraro}, which shows that if $E_{/F}$ is an elliptic curve, then $H^1(F_\Sigma/F,E[p^\infty])$ does not have proper $\Lambda$-submodules of finite index. As a consequence, we will see that $S_\Gamma=0$. 

In order to prove this non-existence result, we need four lemmas.

\begin{lemma} \label{g surj}
The map $g$ is surjective. 
\end{lemma}

\begin{proof} Using the fact that $\Gamma$ has cohomological dimension $1$, one can check that each of the local components defining $g$ is surjective. \end{proof}

The second lemma we are interested in is a generalization of \cite[Lemma 2.3]{greenberg-cetraro}.

\begin{lemma} \label{geenberg2.3}
Let $v_n$ be the prime of $F_n$ above $p$, denote by $F_{n,v_n}$ the completion of $F_n$ at $v_n$ and set $G_{F_{n,v_n}}:=\Gal(\bar\Q_p/F_{n,v_n})$. Let $\psi:G_{F_v}\rightarrow\mathcal{O}^\times$ be a character and write $(K/\cO)(\psi)$ for the group $K/\cO$ with $G_{F_v}$-action
given by $\psi$. If ${\psi|}_{G_{F_{n,v_n}}}$ is non-trivial and does not coincide with the cyclotomic character, then the $\cO$-corank of $H^1\bigl(F_{n,v_n},(K/\cO)(\psi)\bigr)$ is $[F_{n,v_v}:\Q_p]$.  
\end{lemma}

\begin{proof} To simplify the notation, put $C:=(K/\cO)(\psi)$, $M:=F_{n,v_n}$ and $G_M:=\Gal(\bar\Q_p/M)$. Let $T(C)$ be the $p$-adic Tate module of $C$, set $V(C):=T(C)\otimes_\cO K$ and define
\[ h_i:=\dim_{\Q_p}\Bigl(H^i\bigl(M,V(C)\bigr)\!\Bigr) \] 
for $i=0,1,2$. Since $\dim_{\Q_p}\bigl(V(C)\bigr)=[{K}:\Q_p]$, the Euler characteristic of $V(C)$ over $M$ is given by the formula 
\begin{equation} \label{euler characteristic}
h_0-h_1+h_2=-[M:\Q_p]\cdot[{K}:\Q_p].
\end{equation} 
Let $\mathcal{K}(\psi)$ denote $\mathcal{K}$ (viewed as a $\mathcal{K}$-vector space over itself) equipped with the $G_{F_v}$-action induced by $\psi$. There is an isomorphism $V(C)\simeq\mathcal{K}(\psi)$, so $H^0\bigl(M,V(C)\bigr)=0$, where the restriction of $\psi$ to $G_M$ is non-trivial by assumption. Let $\Q_p(1)$ be the Tate twist of $\Q_p$ and define 
\[ V(C)^*:=\Hom_{\Q_p}\bigl(V(C),\Q_p(1)\bigr). \] 
The $H^0\bigl(M,V(C)^*\bigr)=0$ because $V(C)\simeq\mathcal{K}(\psi)$ and $\psi$ does not coincide with the cyclotomic character. Poitou--Tate duality implies that $H^2\bigl(M,V(C)\bigr)$ is dual to $H^0\bigl(M,V(C)^*\bigr)$. Since $H^0\bigl(M,V(C)^*\bigr)=0$, we deduce that $H^2\bigl(M,V(C)\bigr)=0$.  
Therefore, $h_0=h_2=0$ and it follows from \eqref{euler characteristic} that $h_1=[M:\Q_p]\cdot[{K}:\Q_p]$. 

We want to prove that the $\Z_p$-corank of $H^1(M,C)$ is equal to $h_1$. This completes the proof of the lemma because, by what we have just shown, the $\Z_p$-corank of $H^1(M,C)$ is then equal to $[M:\Q_p]\cdot[K:\Q_p]$ and so its $\cO$-corank must be equal to $[M:\Q_p]$. To prove this claim, note that the short exact sequence 
\[ 0\longrightarrow T(C)\longrightarrow V(C)\longrightarrow C\longrightarrow0 \] 
induces an exact sequence 
\[ H^1\bigl(M,T(C)\bigr)\longrightarrow H^1\bigl(M,V(C)\bigr)\longrightarrow H^1(M,C)\longrightarrow H^2\bigl(M,T(C)\bigr). \]
Set $T(C)^*:=\Hom\bigl(T(C),\Bmu_{p^\infty}\bigr)$. By Poitou--Tate duality, $H^2\bigl(M,T(C)\bigr)$ is identified with $H^0\bigl(M,T(C)^*\bigr)$. Since the action of $G_M$ on $T(C)$ is via $\psi$, we get an isomorphism
\[ H^0\bigl(M,T(C)^*\bigr)\simeq H^0\bigl(M,(K/\cO)(\chi_\cyc\psi^{-1})\bigr), \]
where $(K/\cO)(\chi_\cyc\psi^{-1})$ is $K/\mathcal{O}$ with Galois action twisted by $\chi_\cyc\psi^{-1}$. On the other hand, $\psi\neq\chi_\cyc$ by assumption, so $H^0\bigl(M,T(C)^*\bigr)=0$. Thus, $H^2\bigl(M,V(C)\bigr)$ is also trivial. Now, by \cite[Proposition 2.3]{Tate}, the image of $H^1\bigl(M,T(C)\bigr)$ in $H^1\bigl(M,V(C)\bigr)$ is a lattice, say $P$, so the quotient of $H^1\bigl(M,V(C)\bigr)$ by $P$, which is just 
$H^1(M,C)$ because $H^2\bigl(M,T(C)\bigr)=0$, is isomorphic to $(\Q_p/\Z_p)^{h_1}$. Taking $\Z_p$-duals shows that the $\Z_p$-corank of $H^1(M,C)$ is $h_1$, as was to be proved. \end{proof}

\begin{lemma} \label{finite}
If $v\nmid p$, then $H^0(F_v,A)$ is finite.  
\end{lemma}

\begin{proof} If $H^0(F_v,A)$ is not finite, then there exists a $p$-divisible subgroup $B\subset A$ that is fixed by $G_{F_v}$. Let us choose an element $b\in B[p^M]\smallsetminus\{0\}$ for some integer $M\geq1$. By part (4) of Assumption \ref{ass}, for every integer $m\geq1$ there is a skew-symmetric, $G_F$-equivariant and non-degenerate pairing 
\[ {(\cdot,\cdot)}_m:A[p^m]\times A[p^m]\longrightarrow\Bmu_{p^m}. \]
Choose $c\in A[p^M]$ such that ${(c,b)}_M=\zeta_{p^r}$ for a non-trivial, primitive $p^r$-th root of unity $\zeta_{p^r}$ (so $r\leq M$). Using the fact that the representation $\rho_T$ is continuous and $T/p^mT\simeq A[p^m]$ for all $m\geq1$, there is a finite extension $H/F_v$ with $b,c\in H^0(H,A)$. Since $B$ is divisible, for every $N\geq M$ we can pick an element $b_N\in B[p^N]$ with the property that $p^{N-M}b_N=b$. Then ${(c,b_N)}_N=\zeta_{p^N}$ for some $p^N$-th root of unity $\zeta_{p^N}$ such that $\zeta_{p^{N}}^{N-M}=\zeta_{p^r}$; in particular, $\zeta_{p^N}$ is a primitive $p^{N-M+r}$-th root of unity. By assumption, 
$b_N\in H^0(H,A)$, and then the Galois-equivariance of ${(\cdot,\cdot)}_N$ ensures that 
\[ {(c,b_N)}_N=\sigma(\zeta_{p^{N}}) \]
for all $\sigma\in G_H$. Since $N$ is arbitrary and $\zeta_{p^N}$ is a primitive $p^{N-M+r}$-th root of unity, this implies that $H$ contains the cyclotomic $\Z_p$-extension of $F_v$, which is impossible because the extension $H/F_v$ is finite. \end{proof}

The proof of the next lemma, which follows from \cite[Propositions 1 and 2]{greenberg-iwasawa}, proceeds along the lines of the arguments in \cite[p. 94]{greenberg-cetraro}

\begin{lemma} \label{corank}
The $\Lambda$-corank of $\mathcal{P}^\Sigma_\mathrm{Gr}(A/F_\infty)$ is $(r-r^+)\cdot[F:\Q]=r^-\cdot[F:\Q]$. \end{lemma}

\begin{proof} Let $\mathcal{P}^v_\mathrm{Gr}(A/F_\infty)$ denote the
factor in $\mathcal{P}^\Sigma_\mathrm{Gr}(A/F_\infty)$ corresponding to the place $v$. For $v\nmid p$, $\mathcal{P}_\mathrm{Gr}^v(A/F_\infty)$ is $\Lambda$-cotorsion by \cite[Proposition 1]{greenberg-iwasawa}. This result is clear if the prime $v$ is archimedean, as in this case this module has exponent $2$. For finite primes $v\nmid p$, $\mathcal{P}^v_\mathrm{Gr}(A/F)$ is finite: this is a consequence of \cite[Proposition 2]{greenberg-iwasawa} combined with Lemma \ref{finite} and the isomorphism $A\simeq A^\vee(1)$. Since the map $\mathcal{P}_\mathrm{Gr}^v(A/F)\rightarrow\mathcal{P}_\mathrm{Gr}^v(A/F_\infty)^\Gamma$ is surjective by Lemma \ref{g surj}, $\mathcal{P}_\mathrm{Gr}^v(A/F_\infty)^\Gamma$ is finite as well, which implies that $\mathcal{P}_\mathrm{Gr}^v(A/F_\infty)$ is $\Lambda$-cotorsion; here we are using \cite[(2), p. 79]{greenberg-cetraro} and the fact that all the primes in $\Sigma$ are finitely decomposed in $F_\infty$. 

Now let $v\,|\,p$ be a place of $F_\infty$ and let $\Gamma_v$ be the corresponding decomposition group. By \cite[Proposition 1]{greenberg-iwasawa}, the $\Z_p[\![\Gamma_v]\!]$-corank of $H^1(F_{\infty,v},A)$ is $r\cdot[F_v:\Q_p]\cdot[K:\Q_p]$ (to apply this result, note that $K/\cO\simeq (\Q_p/\Z_p)^{[K:\Q_p]}$ as groups). Thus, the $\mathcal{O}[\![\Gamma_v]\!]$-corank of $H^1(F_{\infty,v},A)$ is $r\cdot[F_v:\Q_p]$. There is an isomorphism $A^+\simeq(K/\cO)^{r^+}$ of $\cO$-modules and the action of $G_{F_v}$ on $A^+$ is via the unramified characters $\eta_1,\dots,\eta_{r^+}$, which are non-trivial thanks to condition \eqref{condition e} in Assumption \ref{ass}. If $v_n$ denotes the prime of $F_n$ above $p$, then 
Lemma \ref{geenberg2.3} guarantees that the $\cO$-corank of $H^1(F_{n,v_n},A^+)$ is $r^+\cdot[F_{n,v_n}:\Q_p]$. This implies that $H^1(F_{\infty,v},A^+)$ has $\cO[\![\Gamma_v]\!]$-corank $r^+\cdot[F_v:\Q_p]$. Consequently, the $\mathcal{O}[\![\Gamma_v]\!]$-corank of the quotient 
$H^1(F_{\infty,v},A)\big/H^1_\mathrm{ord}(F_{\infty,v},A)$ is $(r-r^+)\cdot[F_v:\Q_p]$. Finally, the well-known formula 
\[ \sum_{v\mid p}[F_v:\Q_p]=[F:\Q] \]
(see, \emph{e.g.}, \cite[p. 39, Corollary 1]{lang-ANT}) concludes the proof. \end{proof}

Now we can prove a result establishing, in particular, the non-existence of proper $\Lambda$-submodules of finite index of $H^1(F_\Sigma/F_\infty,A)$.

\begin{proposition} \label{submodules}
The $\Lambda$-module $H^1(F_\Sigma/F_\infty,A)$ is cofinitely generated of rank $r^-\cdot[F:\Q]$ and has no proper $\Lambda$-submodules of finite index.
\end{proposition}

\begin{proof} Since $\Sel_\mathrm{Gr}(A/F_\infty)$ is $\Lambda$-cotorsion by Proposition \ref{lemma4.1}, it follows from Lemma \ref{corank} that the $\Lambda$-corank of $H^1(F_\Sigma/F_\infty,A)$ is at most $r^-\cdot[F:\Q]$. By \cite[Proposition 3 and (34)]{greenberg-iwasawa}, the $\Lambda$-corank of $H^1(F_\Sigma/F_\infty,A)$ is $r^-\cdot[F:\Q]$ and the $\Lambda$-corank of $H^2(F_\Sigma/F_\infty,A)$ is $0$. Since  $H^2(F_\Sigma/F_\infty,A)$ is $\Lambda$-cofree by \cite[Proposition 4]{greenberg-iwasawa}, we deduce that $H^2(F_\Sigma/F_\infty,A)=0$. The lemma is a consequence of \cite[Proposition 5]{greenberg-iwasawa}. \end{proof}

\subsection{Triviality of coinvariants} 

Recall that $S=\Sel_\mathrm{Gr}(A/F_\infty)$. Our goal here is to prove that $S_\Gamma=0$. First of all, we need a lemma on the interaction between Pontryagin duals and torsion submodules, which is valid in a slightly more general context.

\begin{lemma} \label{C1} 
Let $N$ be a compact $\Lambda$-algebra, let $I$ be an ideal of $\Lambda$ and write $N[I]$ for the $I$-torsion submodule of $N$. There is an isomorphism of $\Lambda$-modules 
\[ N^\vee/IN^\vee\simeq N[I]^\vee. \]
\end{lemma}

\begin{proof} Write $I=(x_1,\dots,x_n)$ and consider the map 
\[ \xi:N\longrightarrow\prod_{i=1}^nx_iN,\qquad n\longmapsto{(x_in)}_{i=1,\dots,n}, \] 
whose kernel is equal to $N[I]$. If $i:N[I]\hookrightarrow N$ denotes inclusion then Pontryagin duality gives an exact sequence of $\Lambda$-modules
\begin{equation} \label{exact-dual-eq}
\Bigg(\prod_{i=1}^nx_iN\Bigg)^{\!\!\!\vee}\overset{\xi^\vee}\longrightarrow N^\vee\overset {i^\vee} \longrightarrow N[I]^\vee\longrightarrow0; 
\end{equation}
here the surjectivity of $i^\vee$ is a consequence of $N[I]$ being closed in $N$, hence compact. On the other hand, sending $(\varphi_1,\dots,\varphi_n)$ to $\sum_i\varphi_i$ gives an isomorphism between $\prod_i(x_iN)^\vee$ and $\bigl(\prod_ix_iN\bigr)^{\!\vee}$, so we can rewrite \eqref{exact-dual-eq} as
\begin{equation} \label{exact-dual-eq2}
\prod_{i=1}^n(x_iN)^\vee\overset{\xi^\vee}\longrightarrow N^\vee\overset {i^\vee} \longrightarrow N[I]^\vee\longrightarrow0. 
\end{equation}
In light of \eqref{exact-dual-eq2}, we want to check that ${\rm im}(\xi^\vee)=IN^\vee$. First of all, let $\varphi=\sum_{i=1}^nx_i\varphi_i\in IN^\vee$, with $\varphi_i\in N^\vee$ for all $i=1,\dots,n$. Then $\varphi=\xi^\vee\bigl(({\varphi_1|}_{x_1N},\dots,{\varphi_n|}_{x_nN})\bigr)$, which shows that $\varphi\in{\rm im}(\xi^\vee)$. Conversely, let $\varphi\in{\rm im}(\xi^\vee)$; by definition, for every $i=1,\dots,n$ there exists $\varphi_i\in(x_iN)^\vee$ such that $\varphi=\xi^\vee\bigl((\varphi_1,\dots,\varphi_n)\bigr)$. For every $i$, the $\Lambda$-module $x_iN$ is compact because $N$ is, hence the inclusion $x_iN\hookrightarrow N$ gives a surjection $N^\vee\twoheadrightarrow(x_iN)^\vee$. Now for every $i=1,\dots,n$ choose a lift $\psi_i\in N^\vee$ of $\varphi_i$. It follows that $\varphi=\sum_{i=1}^nx_i\psi_i\in IN^\vee$, and the proof is complete. \end{proof} 

The vanishing of $S_\Gamma$ will be a consequence of the following result.

\begin{proposition} \label{greenberg-prop2.16}
${H^1(F_\Sigma/F_\infty,A)}_\Gamma=0$.
\end{proposition}

\begin{proof} By Proposition \ref{submodules}, it suffices to show that ${H^1(F_\Sigma/F_\infty,A)}_\Gamma$ is finite. Let us denote by ${H^1(F_\Sigma/F_\infty,A)}_{\Lambda\text{-div}}$ the maximal $\Lambda$-divisible submodule of $H^1(F_\Sigma/F_\infty,A)$ and define the $\Lambda$-module $Q$ via the short exact sequence 
\begin{equation} \label{div-eq}
0\longrightarrow {H^1(F_\Sigma/F_\infty,A)}_{\Lambda\text{-div}}\longrightarrow H^1(F_\Sigma/F_\infty,A)\longrightarrow Q\longrightarrow 0. 
\end{equation}
The Pontryagin dual $Q^\vee$ of $Q$ is the torsion $\Lambda$-submodule of the Pontryagin dual $Y$ of $H^1(F_\Sigma/F_\infty,A)$; it follows from Proposition \ref{submodules} that there is a pseudoisomorphism
\begin{equation} \label{y-eq}
Y\sim\Lambda^{r^-\cdot[F:\Q]}\oplus Q^\vee, 
\end{equation}
and $Q$ is a cofinitely generated cotorsion $\Lambda$-module. Set $M:=H^1(F_\Sigma/F_\infty,A)$ and fix, as before, a topological generator $\gamma$ of $\Gamma$. Combining the vanishing of $H^2(\Gamma,N)$ for all torsion discrete $\Gamma$-modules $N$ (see, \emph{e.g.}, \cite[Corollary 4.27]{hida-modular}) with the identifications $H^1(\Gamma,N)=N/(\gamma-1)N=N_\Gamma$ for every $\Gamma$-module $N$, exact sequence \eqref{div-eq} yields an exact sequence
\begin{equation} \label{div-eq2}
0\longrightarrow M_{\Lambda\text{-div}}^\Gamma\longrightarrow M^\Gamma\longrightarrow Q^\Gamma\longrightarrow{(M_{\Lambda\text{-div}})}_\Gamma\longrightarrow M_\Gamma\longrightarrow 
Q_\Gamma\longrightarrow 0 
\end{equation}
in Galois cohomology. Since $(M_{\Lambda\text{-div}})_\Gamma=0$ because $M_{\Lambda\text{-div}}$ is $\Lambda$-divisible, it follows from \eqref{div-eq2} that $M_\Gamma\simeq Q_\Gamma$. Therefore, it is enough to show that $Q_\Gamma$ is finite. Furthermore, since $Q$ is a finitely generated cotorsion $\Lambda$-module, the exact sequence 
\begin{equation} \label{q-eq}
0\longrightarrow Q^\Gamma\longrightarrow Q\xrightarrow{(\gamma-1)\cdot}Q\longrightarrow Q_\Gamma\longrightarrow0 
\end{equation}
shows that $Q^\Gamma$ and $Q_\Gamma$ have the same $\mathcal{O}$-corank. Thus, we are reduced to proving that the $\cO$-corank of $Q^\Gamma$ is $0$.

Lemma \ref{corank} implies that the $\mathcal{O}$-corank of $\mathcal{P}_\mathrm{Gr}^\Sigma(A/F_\infty)^\Gamma$ is $r^-\cdot[F:\Q]$. By Lemma \ref{g surj} and the finiteness of $\ker(g)$, the $\cO$-corank of $\mathcal{P}_\mathrm{BK}^\Sigma(A/F)$ is $r^-\cdot[F:\Q]$ as well. In addition, by Lemma \ref{surjectivity lemma}, the surjection $\delta_\Sigma:H^1(F_\Sigma/F,A)\twoheadrightarrow\mathcal{P}_\mathrm{BK}^\Sigma(A/F)$ has the finite group $\Sel_{\mathrm{BK}}(A/F)$ as its kernel, hence
\begin{equation} \label{corank-F-eq}
\corank_\cO\bigl(H^1(F_\Sigma/F,A)\bigr)=r^-\cdot[F:\Q].
\end{equation}
Now consider the inflation-restriction exact sequence 
\[ 0\longrightarrow H^1\bigl(\Gamma,A(F_\infty)\bigr)\longrightarrow H^1(F_\Sigma/F,A)\overset{\theta}\longrightarrow H^1(F_\Sigma/F_\infty,A)^\Gamma\longrightarrow H^2\bigl(\Gamma,A(F_\infty)\bigr). \] 
The map $\theta$ is an isomorphism because, by Lemma \ref{greenberg-lemma2.8}, $A(F_\infty)=0$. Thus, in particular, we get an equality
\begin{equation} \label{coranks-eq}
\corank_\cO\bigl(H^1(F_\Sigma/F,A)\bigr)=\corank_\cO\bigl(H^1(F_\Sigma/F_\infty,A)^\Gamma\bigr).
\end{equation}
Taking $I=(\gamma-1)$ and $N=Q$ in Lemma \ref{C1}, we obtain an isomorphism of $\Lambda$-modules ${(Q^\vee)}_\Gamma\simeq\bigl(Q^\Gamma\bigr)^\vee$. On the other hand, exact sequence \eqref{q-eq} with $Q^\vee$ in place of $Q$ shows that $(Q^\vee)^\Gamma$ and ${(Q^\vee)}_\Gamma$ have the same $\cO$-rank. We surmise that 
\begin{equation} \label{O-eq}
\rank_\cO\bigl((Q^\vee)^\Gamma\bigr)=\rank_\cO\bigl({(Q^\vee)}_\Gamma\bigr)=\rank_\cO\bigl((Q^\Gamma)^\vee\bigr)=\corank_\cO\bigl(Q^\Gamma\bigr).
\end{equation}
Analogously, since $Y=H^1(F_\Sigma/F_\infty,A)^\vee$, we get an equality
\begin{equation} \label{H-eq}
\rank_\cO\bigl(Y^\Gamma\bigr)=\corank_\cO\bigl(H^1(F_\Sigma/F_\infty,A)^\Gamma\bigr).
\end{equation}
In light of \eqref{O-eq} and \eqref{H-eq}, pseudoisomorphism \eqref{y-eq} ensures that
\begin{equation} \label{corank-eq}
\corank_\cO\bigl(H^1(F_\Sigma/F_\infty,A)^\Gamma\bigr)=r^-\cdot[F:\Q]+\corank_\cO\bigl(Q^\Gamma\bigr).
\end{equation} 
Finally, combining \eqref{corank-F-eq}, \eqref{coranks-eq} and \eqref{corank-eq} gives $\corank_\cO\bigl(Q^\Gamma\bigr)=0$, as was to be shown. \end{proof}

Now we can turn to the vanishing result that will play a crucial role in the proof of our main theorem.

\begin{corollary} \label{greenberg-coro2.17} 
$S_\Gamma=0$. 
\end{corollary}

\begin{proof} With notation as in \S \ref{surjectivity-subsec}, set $\tilde{\mathcal{P}}^\Sigma_\mathrm{Gr}(A/F_\infty):=\im(\delta_{\Sigma,\infty})\subset\mathcal{P}^\Sigma_\mathrm{Gr}(A/F_\infty)$, so that there is a short exact sequence
\begin{equation} \label{tilde-P-eq}
0\longrightarrow S\longrightarrow H^1(F_\Sigma/F_\infty,A)\xrightarrow{\delta_{\Sigma,\infty}}\tilde{\mathcal{P}}^\Sigma_\mathrm{Gr}(A/F_\infty)\longrightarrow0 
\end{equation}
of $\Gamma$-modules. Diagram \eqref{greenberg-diagram} and Lemma \ref{g surj} imply that $\delta^\Gamma_{\infty,\Sigma}$ is surjective. It follows that $\delta^\Gamma_{\infty,\Sigma}$ induces a surjection
\[ \delta^\Gamma_{\infty,\Sigma}:H^1(F_\Sigma/F_\infty,A)^\Gamma\longepi\tilde{\mathcal{P}}^\Sigma_\mathrm{Gr}(A/F_\infty)^\Gamma, \]
which we denote by the same symbol. We can extract from the long exact sequence in cohomology associated with \eqref{tilde-P-eq} an exact sequence
\[ H^1(F_\Sigma/F_\infty,A)^\Gamma\xrightarrow{\delta^\Gamma_{\infty,\Sigma}}\tilde{\mathcal{P}}^\Sigma_\mathrm{Gr}(A/F_\infty)^\Gamma\longrightarrow S_\Gamma\longrightarrow{H^1(F_\Sigma/F_\infty,A)}_\Gamma. \]
Since $\delta^\Gamma_{\infty,\Sigma}$ is surjective, we deduce that $S_\Gamma$ embeds into ${H^1(F_\Sigma/F_\infty,A)}_\Gamma$, and the triviality of $S_\Gamma$ follows from Proposition \ref{greenberg-prop2.16}. \end{proof}

\begin{remark}
A result analogous to Corollary \ref{greenberg-coro2.17} in an anticyclotomic imaginary quadratic setting can be found in \cite[Lemma 3.3.5]{JSW}.
\end{remark}

\section{Main result} \label{main-sec}

Putting together the results we have collected so far, we obtain the main theorem of this paper. For the convenience of the reader, we recall our setting. 

Let $F$ be a number field, let $\cO$ be the valuation ring of a finite field extension $K$ of $\Q_p$ and let $T$ be a free $\cO$-module that is equipped with a continuous action of the absolute Galois group of $F$ satisfying Assumption \ref{ass}. Let $F_\infty/F$ be the cyclotomic $\Z_p$-extension of $F$ and let $\Lambda:=\mathcal{O}[\![\Gamma]\!]$ be the associated Iwasawa algebra, where $\Gamma:=\Gal(F_\infty/F)\simeq\Z_p$. Let $A:=T\otimes_{\cO}(K/\cO)$, let $\Sel_{\mathrm{Gr}}(A/F_\infty)$ be the Greenberg Selmer group of $A$ over $F_\infty$ and let $\Sel_\mathrm{BK}(A/F)$ the Bloch--Kato Selmer group of $A$ over $F$. Finally, let $c_v(A)$ be the Tamagawa number of $A$ at a prime $v$ of $F$ and denote by $\Sigma$ the (finite) set of primes $v$ of $F$ such that either $v$ is archimedean or $v$ lies above $p$ or $A$ is ramified at $v$. Notice that the finiteness of $\Sel_\mathrm{BK}(A/F)$ in the statement below is Assumption \ref{ass2}.

\begin{theorem} \label{main theorem}
Suppose that $\Sel_\mathrm{BK}(A/F)$ is finite. Then
\begin{enumerate} 
\item $\Sel_\mathrm{Gr}(A/F_\infty)$ is $\Lambda$-cotorsion; 
\item if $\mathcal{F}$ is the characteristic power series of the Pontryagin dual of $\Sel_\mathrm{Gr}(A/F_\infty)$, then $\mathcal{F}(0)\neq0$;
\item there is an equality 
\[ \#\bigl(\cO/\mathcal{F}(0)\cdot\cO\bigr)=\#\Sel_\mathrm{BK}(A/F)\cdot\prod_{\substack{v\in\Sigma\\v\nmid p}}c_v(A). \]
\end{enumerate}
\end{theorem}

\begin{proof} Parts (1) and (2) are \eqref{1 and 2}, while part (3) follows by combining \eqref{greenberg-eq2.2}, Proposition \ref{greenberg1} and Corollary \ref{greenberg-coro2.17}. \end{proof} 

\begin{remark}
The equality in part (3) of Theorem \ref{main theorem} can be equivalently formulated as
\[ \length_\cO\bigl(\cO/\mathcal{F}(0)\cdot\cO\bigr)=\length_\cO\bigl(\Sel_\mathrm{BK}(A/F)\bigr)\cdot\ord_\p\biggl(\prod\nolimits_{\substack{v\in\Sigma\\v\nmid p}}\,c_v(A)\!\biggr), \]
where $\length_\cO(\star)$ denotes the length of the $\cO$-module $\star$ and $\ord_\p$ is the $\p$-adic valuation.
\end{remark}

\begin{remark}
An analogue of part (3) of Theorem \ref{main theorem} when $F$ is an imaginary quadratic field and $F_\infty$ is the anticyclotomic $\Z_p$-extension of $F$ is provided by \cite[Theorem 3.3.1]{JSW}.
\end{remark}

\bibliographystyle{amsplain}
\bibliography{Iwasawa}

\providecommand{\bysame}{\leavevmode\hbox to3em{\hrulefill}\thinspace}
\providecommand{\MR}{\relax\ifhmode\unskip\space\fi MR }
\providecommand{\MRhref}[2]{%
  \href{http://www.ams.org/mathscinet-getitem?mr=#1}{#2}
}
\providecommand{\href}[2]{#2}
\begin{thebibliography}{10}

\bibitem{BH}
P.~N. Balister and S.~Howson, \emph{Note on {N}akayama's lemma for compact
  {$\Lambda$}-modules}, Asian J. Math. \textbf{1} (1997), no.~2, 224--229.

\bibitem{BD-IMC}
M.~Bertolini and H.~Darmon, \emph{Iwasawa's {M}ain {C}onjecture for elliptic
  curves over anticyclotomic {$\mathbb Z_p$}-extensions}, Ann. of Math. (2)
  \textbf{162} (2005), no.~1, 1--64.

\bibitem{BK}
S.~Bloch and K.~Kato, \emph{{$L$}-functions and {T}amagawa numbers of motives},
  The {G}rothendieck {F}estschrift, {V}ol.\ {I}, Progr. Math., vol.~86,
  Birkh\"auser Boston, Boston, MA, 1990, pp.~333--400.

\bibitem{Del-Bourbaki}
P.~Deligne, \emph{Formes modulaires et repr\'esentations {$\ell$}-adiques},
  S\'eminaire {B}ourbaki. {V}ol. 1968/69: {E}xpos\'es 347--363, Lecture Notes
  in Math., vol. 175, Springer, Berlin, 1971, pp.~139--172.

\bibitem{Fischman}
A.~Fischman, \emph{On the image of {$\Lambda$}-adic {G}alois representations},
  Ann. Inst. Fourier (Grenoble) \textbf{52} (2002), no.~2, 351--378.

\bibitem{Flach}
M.~Flach, \emph{A generalisation of the {C}assels--{T}ate pairing}, J. Reine
  Angew. Math. \textbf{412} (1990), 113--127.

\bibitem{greenberg-iwasawa}
R.~Greenberg, \emph{Iwasawa theory for {$p$}-adic representations}, Algebraic
  number theory, Adv. Stud. Pure Math., vol.~17, Academic Press, Boston, MA,
  1989, pp.~97--137.

\bibitem{greenberg-cetraro}
\bysame, \emph{Iwasawa theory for elliptic curves}, Arithmetic theory of
  elliptic curves ({C}etraro, 1997), Lecture Notes in Math., vol. 1716,
  Springer, Berlin, 1999, pp.~51--144.

\bibitem{hida-modular}
H.~Hida, \emph{Modular forms and {G}alois cohomology}, Cambridge Studies in
  Advanced Mathematics, vol.~69, Cambridge University Press, Cambridge, 2000.

\bibitem{Howard-inv}
B.~Howard, \emph{Variation of {H}eegner points in {H}ida families}, Invent.
  Math. \textbf{167} (2007), no.~1, 91--128.

\bibitem{JSW}
D.~Jetchev, C.~Skinner, and X.~Wan, \emph{The {B}irch and {S}winnerton-{D}yer
  formula for elliptic curves of analytic rank one}, Camb. J. Math. \textbf{5}
  (2017), no.~3, 369--434.

\bibitem{lang-ANT}
S.~Lang, \emph{Algebraic number theory}, second ed., Graduate Texts in
  Mathematics, vol. 110, Springer-Verlag, New York, 1994.

\bibitem{LV}
M.~Longo and S.~Vigni, \emph{A refined {B}eilinson--{B}loch conjecture for
  motives of modular forms}, Trans. Amer. Math. Soc. \textbf{369} (2017),
  no.~10, 7301--7342.

\bibitem{LV-kyoto}
\bysame, \emph{Kolyvagin systems and {I}wasawa theory of generalized {H}eegner
  cycles}, Kyoto J. Math. \textbf{59} (2019), no.~3, 717--746.

\bibitem{LV-BSD}
\bysame, \emph{Kolyvagin's conjecture and the {B}loch--{K}ato formula for
  modular forms}, preprint (2020).

\bibitem{Nek}
J.~Nekov{\'a}{\v{r}}, \emph{Kolyvagin's method for {C}how groups of
  {K}uga--{S}ato varieties}, Invent. Math. \textbf{107} (1992), no.~1, 99--125.

\bibitem{Nek-Selmer}
\bysame, \emph{Selmer complexes}, Ast\'erisque (2006), no.~310.

\bibitem{Ochiai}
T.~Ochiai, \emph{Control theorem for {B}loch--{K}ato's {S}elmer groups of
  {$p$}-adic representations}, J. Number Theory \textbf{82} (2000), no.~1,
  69--90.

\bibitem{panchishkin}
A.~A. Panchishkin, \emph{Motives over totally real fields and {$p$}-adic
  {$L$}-functions}, Ann. Inst. Fourier (Grenoble) \textbf{44} (1994), no.~4,
  989--1023.

\bibitem{ribet2}
K.~A. Ribet, \emph{On {$l$}-adic representations attached to modular forms
  {II}}, Glasgow Math. J. \textbf{27} (1985), 185--194.

\bibitem{Rubin-ES}
K.~Rubin, \emph{Euler systems}, Annals of Mathematics Studies, vol. 147,
  Princeton University Press, Princeton, NJ, 2000.

\bibitem{Rubin}
\bysame, \emph{Euler systems and {K}olyvagin systems}, Arithmetic of
  {$L$}-functions, IAS/Park City Math. Ser., vol.~18, Amer. Math. Soc.,
  Providence, RI, 2011, pp.~449--499.

\bibitem{Serre}
J.-P. Serre, \emph{Linear representations of finite groups}, Graduate Texts in
  Mathematics, vol.~42, Springer-Verlag, New York, 1977.

\bibitem{Serre-Cheb}
\bysame, \emph{Quelques applications du th\'eor\`eme de densit\'e de
  {C}hebotarev}, Inst. Hautes \'Etudes Sci. Publ. Math. \textbf{54} (1981),
  323--401.

\bibitem{SU}
C.~Skinner and E.~Urban, \emph{The {I}wasawa {M}ain {C}onjectures for {$\rm
  GL_2$}}, Invent. Math. \textbf{195} (2014), no.~1, 1--277.

\bibitem{Tate}
J.~Tate, \emph{Relations between {$K_{2}$} and {G}alois cohomology}, Invent.
  Math. \textbf{36} (1976), no.~1, 257--274.

\end{thebibliography}
\end{document}